\documentclass[12pt]{article}                 % Specifies the document

\usepackage{algorithm,algorithmic}
\usepackage{graphicx}
\usepackage{float}
\usepackage{fullpage}

\usepackage[english]{babel}
\usepackage{amsmath}
\usepackage{amsthm}
\usepackage{amssymb}
\usepackage{subfigure}
\newlength{\captsize} \let\captsize=\footnotesize
\newlength{\captwidth} 

\def\defemb#1#2{\expandafter\def\csname #1\endcsname
                              {\relax\ifmmode #2\else\hbox{$#2$}\fi}}

\defemb{cT}{{\cal T}}
\defemb{cM}{{\cal M}}

\def\ba{\begin{align}}
\def\ea{\end{align}}

\newcommand{\vectornorm}[1]{\left|\left|#1\right|\right|}

\newcommand{\abs} [1]{\left\lvert #1 \right\rvert}

\newcommand{\beq}[1]{ \begin{equation}\label{#1} }
\newcommand{\eeq}{\end{equation}}
\newcommand{\beb}{ \begin{bmatrix}}
\newcommand{\eeb}{\end{bmatrix}}
\newcommand{\bea}[1]{ \begin{eqnarray}\label{#1} }
\newcommand{\eea}{\end{eqnarray}}

\def\ll{[\![}
\def\rr{]\!]}
\def\Den#1{\relax\ifmmode \ll #1\rr \else\hbox{$\ll #1\rr$}\fi}

\defemb{cA}{{\cal A}}
\defemb{cB}{{\cal B}}
\defemb{cC}{{\cal C}}
\defemb{cI}{{\cal I}}
\defemb{cN}{{\cal N}}
\defemb{cT}{{\cal T}}
\defemb{cV}{{\cal V}}

\newtheorem{definition}{Definition}[section]
\newtheorem{theorem}[definition]{Theorem}
\newtheorem{lemma}[definition]{Lemma}

\newtheorem{corollary}[definition]{Corollary}

\title{Matrix Compression using the Nystr\"{o}m Method}
\author{Arik Nemtsov$^1$ ~~~Amir Averbuch$^1$ ~~~Alon Schclar$^2$ \footnote{corresponding author}\\
${}^1$ School of Computer Science, Tel Aviv University, Tel Aviv 69978\\
${}^2$ School of Computer Science, \\
The Academic College of Tel Aviv-Yaffo ,Tel Aviv, 61083
}
\date{}

\begin{document}
\maketitle

\begin{abstract}

The Nystr\"{o}m method is routinely used for out-of-sample extension
of kernel matrices. We describe how this method can be applied to
find the singular value decomposition (SVD) of general matrices and
the eigenvalue decomposition (EVD) of square matrices. We take as an
input a matrix $M\in \mathbb{R}^{m\times n}$, a user defined integer
$s\leq min(m,n)$ and $A_M \in \mathbb{R}^{s\times s}$, a matrix
sampled from the columns and rows of $M$. These are used to construct an
approximate rank-$s$ SVD of $M$ in
$O\left(s^2\left(m+n\right)\right)$ operations. If $M$ is square,
the rank-$s$ EVD can be similarly constructed in $O\left(s^2
n\right)$ operations. Thus, the matrix $A_M$ is a compressed version of $M$.
We discuss the choice of $A_M$ and propose an algorithm that selects a good initial sample for a pivoted version of $M$. The proposed algorithm performs well for general matrices and kernel matrices whose spectra exhibit fast decay.

\end{abstract}
\par\noindent
Keywords: \{Compression, SVD, EVD, Nystr\"{o}m, out-of-sample extension\}\\
%2000 MSC: \{65F15, 65F30, 65F50\}

\section{ Introduction }

Low rank approximation of linear operators is an important problem
in the areas of scientific computing and statistical analysis.
Approximation reduces storage requirements for large datasets and
improves the runtime complexity of algorithms operating on the
matrix. When the matrix contains affinities between elements, low
rank approximation can be used to reduce the dimension of the
original problem (\cite{LLE,MDS,DiffusionMaps}) and to eliminate
statistical noise (\cite{PCA}).

Our approach involves the choice of a small sub-sample from the
matrix, followed by the application of the Nystr\"{o}m method for
out-of-sample extension. The Nystr\"{o}m method
(\cite{BakerNystrom}), which originates from the field of integral
equations, is a way of discretizing an integral equation using a
simple quadrature rule. When given an eigenfunction problem of the
form
\[\lambda f(x)=\int^b_a{M\left(x,y\right)f\left(y\right)dy}, \]
the Nystr\"{o}m method employs a set of $s$ sample points $y_1,\dots
,y_s$ that approximate $f(x)$ as
\[\lambda \tilde{f}\left(x\right)\triangleq \frac{b-a}{s}\sum^s_{j=1}{M\left(x,y_j\right)f\left(y_j\right)} .\]

In recent years, the Nystr\"{o}m method has gained widespread use in
the field of spectral clustering. It was first popularized by
\cite{FirstKernelNystrom} for sparsifying kernel matrices by
approximating their entries. The matrix completion approach of
\cite{Chung} also enables the approximation of eigenvectors. It was
now possible to use the Nystr\"{o}m method in order to speed up
algorithms that require the spectrum of a kernel matrix. Over time,
Nystr\"{o}m based out-of-sample extensions have been developed for a
wide range of spectral methods, including Normalized-Cut
(\cite{SpatioTemportalNystrom,NystromIndefKernel}), Geometric
Harmonics (\cite{GeometricHarmonics}) and others
(\cite{NystromKernelPca}).

 Other noteworthy methods for speeding up kernel based
algorithms, which are not applicable to the proposed setting of this
paper, are based on sampling \cite{A1}, convex optimization
\cite{A2} and integral equations. ACA \cite{A3,A4} is an important
example in the latter category. ACA can be regarded as an efficient
replacement of the SVD which is tailored to asymptotically smooth
kernels. The kernel function itself is not required. ACA uses only
few of the original entries for the approximation of the whole
matrix and it was shown to have exponential convergence when used as
part of the Nystr\"{o}m method.

In this paper, we present two extensions of the matrix completion
approach of \cite{Chung}. These allow us to form the SVD and EVD of
a general matrix through the application of the Nystr\"{o}m method
on a previously chosen sample.

In addition, we present a novel algorithm for selecting the initial
sample to be used with the Nystr\"{o}m method. Our algorithm is
applicable to general matrices whereas previous methods focused on
kernel matrices. The algorithm uses a pre-existing low-rank
decomposition of the input matrix. We show that our sample choice
reduces the Nystr\"{o}m approximation error.

The paper is organized as follows: Section \ref{sec:preliminaries}
describes the basic Nystr\"{o}m matrix form and the methods of
\cite{Chung} for finding the EVD of a Nystr\"{o}m approximated
symmetric matrix. Section \ref{sec:nystrom_like_svd} outlines a
Nystr\"{o}m-like method for out-of-sample extension of general
matrices, starting with the SVD of a sample matrix. In section
\ref{sec:mat_gen_decomp} we describe procedures that explicitly
generate the canonical SVD and EVD forms for general matrices.
Section \ref{sec:sample_choice} introduces the problem of sample
choice and presents results that bound the accuracy of the algorithm
in section \ref{sec:algorithm}. Section \ref{sec:algorithm} presents
our sample selection algorithm and analyzes its complexity.
Experimental results on general and kernel matrices are presented in
section \ref{sec:experiments}.

\section{ Preliminaries }
\label{sec:preliminaries}

\subsection{ Square Nystr\"{o}m Matrix Form }
\label{sec:matform}

Let $M\in {\mathbb R}^{n\times n}$ be a square matrix. We assume
that the $M$ can be decomposed as
\begin{equation}\label{eq:assym_m_decomposition}
M= \left[ \begin{array}{cc}
            A_M & B_M \\
            F_M & C_M
        \end{array} \right]
\end{equation}
where $A_M\in {\mathbb R}^{s\times s}, B_M\in {\mathbb R}^{s \times
\left(n-s\right)}, F_M\in {\mathbb R}^{\left(n-s\right)\times s}$
and $C_M\in {\mathbb R}^{\left(n-s\right)\times \left(n-s\right)}$.
The matrix $A_M$ is designated to be our sample matrix. The size of
our sample is $s$, which is the size of $A_M$.

Let $U\Lambda U^{-1}$ be the eigen-decomposition of $A_M$, where
$U\in {\mathbb R}^{s\times s}$ is the eigenvectors matrix and
$\Lambda\in {\mathbb R}^{s\times s}$ is the eigenvalues matrix. Let
$u^i\in {\mathbb R}^s$ be the column eigenvector belonging to
eigenvalue $\lambda_i$. We aim to extend the column eigenvector (the
discrete form of an eigenfunction) to the rest of $M$. Let
$\hat{u}^i = \left[ \begin{array}{cc}u^i & \tilde{u}^i\end{array}
\right]^T\in {\mathbb R}^{n}$ be the extended eigenvector, where
$\tilde{u}^i\in {\mathbb R}^{n-s}$ is the extended part. By applying
the Nystr\"{o}m method to $u^i$, we get the following form for the
$k^{th}$ coordinate in $\hat{u}^i$:
\begin{equation}\label{NYSEQN}
\lambda_i \hat{u}^i_k \simeq \frac{b-a}{s}\sum^s_{j=1}{M_{kj}\cdot
u^i_j}.
\end{equation}
By setting $\left[a,b\right]=[0,1]$ and presenting Eq.
\eqref{NYSEQN} in matrix product form we obtain
\begin{equation}\label{NYSEQN2}
\lambda_i \tilde{u}^i = \frac{1}{s}F_M\cdot u^i.
\end{equation}
This can be done for all the eigenvalues $\{\lambda_i\}_{i=1}^{s}$
of $A_M$. Denote $\tilde{U} = \left[ \begin{array}{ccc}\tilde{u}^1 &
\dots & \tilde{u}^s \end{array}\right]\in {\mathbb R}^{(n-s)\times
s}$. By placing all expressions of the form Eq. \eqref{NYSEQN2} side
by side we have $\tilde{U}\Lambda = F_M U$. Assuming the matrix
$A_M$ has non-zero eigenvalues (we return to this assumption
in section $\ref{sec:nystrom_accuracy}$), we obtain:
\begin{equation}\label{EXTVEC}
\tilde{U}=F_M U {\Lambda }^{-1}.
\end{equation}
Analogically, we can derive a matrix representation for extending
the left eigenvectors of $M$, denoted as $\tilde{V}\in {\mathbb
R}^{s\times n-s}$:
\begin{equation}\label{EXTLEFTVEC}
\tilde{V}={\Lambda }^{-1}U^{-1}B_M.
\end{equation}

Combining Eqs. \eqref{EXTVEC} and \eqref{EXTLEFTVEC} with the
eigenvectors of $A_M$ yields the full left and right approximated
eigenvectors:
\begin{equation}\label{eq:assym_ext_vec}
\hat{U} = \left[ \begin{array}{c}
U \\
F_MU{\Lambda }^{-1} \end{array} \right] ,\ \ \hat{V}=\left[
\begin{array}{cc} U^{-1} & {\Lambda }^{-1}U^{-1}B_M \end{array}
\right].
\end{equation}
The explicit ``Nystr\"{o}m'' representation of $\hat{M}$ becomes:
\begin{equation}\label{eq:assym_matrix_nys_form}
\begin{array}{c}
\hat{M}=\hat{U}\Lambda \hat{V}=\left[ \begin{array}{c}
U \\
F_MU{\Lambda }^{-1} \end{array} \right]\Lambda \left[
\begin{array}{cc} U^{-1} & {\Lambda }^{-1}U^{-1}B_M \end{array}
\right] = \left[ \begin{array}{cc}
A_M & B_M \\
F_M & F_MA^+_MB_M \end{array}
\right] = \\
 \left[ \begin{array}{c}
A_M \\
F_M \end{array} \right]A^+_M\left[ \begin{array}{cc} A_M & B_M
\end{array} \right]
\end{array}
\end{equation}
where $A^+_M$ denotes the pseudo-inverse of $A_M$.

Equation \eqref{eq:assym_matrix_nys_form} shows that the Nystr\"{o}m
extension does not modify $A_M, B_M$ and $F_M$, and that it
approximates $C_M$ by $F_MA^+_MB_M$.

\subsection{ Decomposition of Symmetric Matrices }

The algorithm given in \cite{Chung} is a commonly used method for
SVD approximation of symmetric matrices. For a given matrix, it
computes the SVD of its Nystr\"{o}m approximated form. The SVD and
EVD of a symmetric matrix coincide up to the signs of the singular
(eigen-) values. Therefore the SVD can approximate both simultaneously.
We describe the method of  \cite{Chung} in section
\ref{sec:symmetric_svd_gen sol}.

\subsubsection{ Symmetric Nystr\"{o}m Matrix Form }

When $M$ is symmetric, the matrix $M$ has the decomposition
\begin{equation}\label{SYMDEC}
M=\left[ \begin{array}{cc}
A_M & B_M \\
B^T_M & C_M \end{array} \right]
\end{equation}
where $A_M\in {\mathbb R}^{s\times s},B_M\in {\mathbb R}^{s\times
\left(n-s\right)}$ and $C_M\in {\mathbb R}^{(n-s)\times(n-s)}$. We
replace $F_M$ in Eq. \eqref{eq:assym_m_decomposition} with $B_M^T$.

By using reasoning similar to section \ref{sec:matform}, we can
express the right and left approximated eigenvectors as:
\begin{equation}\label{SYMEXTVEC}
\hat{U}=\left[ \begin{array}{c}
U \\
B^T_MU{\Lambda }^{-1} \end{array} \right],\ \ \hat{V}=\left[
\begin{array}{cc} U^{-1} & {\Lambda }^{-1}U^{-1}B_M \end{array}
\right].
\end{equation}
The explicit ``Nystr\"{o}m'' representation of $\hat{M}$ becomes:
\begin{equation}\label{SYM_EXPLICIT_NYS}
\begin{array}{c}
\hat{M}=\hat{U}\Lambda \hat{V}=\left[ \begin{array}{c}
U \\
B^T_MU{\Lambda }^{-1} \end{array} \right]\Lambda \left[
\begin{array}{cc} U^{-1} & {\Lambda }^{-1}U^{-1}B_M \end{array}
\right]= \left[ \begin{array}{cc}
A_M & B_M \\
B^T_M & B_M^T A^+_MB_M \end{array}
\right]= \\
\left[ \begin{array}{c}
A_M \\
B^T_M \end{array} \right]A^+_M\left[ \begin{array}{cc} A_M & B_M
\end{array} \right].
\end{array}
\end{equation}

\subsubsection{ Construction of SVD for Symmetric $\hat{\mathbf{M}}$ }
\label{sec:symmetric_svd_gen sol}

Our goal is to find the $s$ leading eigenvalues and eigenvectors of
$\hat{M}$ without explicitly forming the entire matrix.

We begin with the decomposition of $M$ as in Eq. \eqref{SYMDEC}. The
approximation technique in \cite{Chung} uses the standard
Nystr\"{o}m method in Eq. \eqref{SYMEXTVEC} to obtain $\hat{U}$.
Then, the algorithm forms the matrix $Z=\hat{U}{\Lambda }^{1/2}$
such that $\hat{M}=ZZ^T=\hat{U}\Lambda {\hat{U}}^T$. The symmetric
$s\times s$ matrix $Z^TZ$ is diagonalized as $F\Sigma F^T$. The
eigenvectors of $\hat{M}$ are given by $U_o=ZF{\Sigma }^{-1/2}$ and
the eigenvalues are given by $\Sigma $. To qualify for use in the
SVD, $U_o$ and $\Sigma$ must meet the following requirements:

\begin{enumerate}
\item The columns of $U_o$ must be orthogonal. Namely, $U^T_oU_o = I$.
\item The SVD form of $U_o$ and $\Sigma$ must form $\hat{M}$. Formally, $\hat{M} = U_o\Sigma U^T_o$.
\end{enumerate}

The following identities can be readily verified using our
expressions for $U_o$ and $\Sigma$:
\begin{enumerate}
\item Bi-orthogonality: $U^T_oU_o={\Sigma }^{-1/2}F^TZ^TZF{\Sigma }^{{\rm -}{\rm 1/2}}={\Sigma }^{{\rm -}{\rm 1/2}}F^T\left(F\Sigma F^T\right)F{\Sigma }^{{\rm -}{\rm 1/2}}=I ;$
\item SVD form: $U_o\Sigma U^T_o=ZF{\Sigma }^{{\rm -}{\rm 1/2}}\cdot \Sigma \cdot {\Sigma }^{{\rm -}{\rm 1/2}}F^TZ^T=ZZ^T=\hat{M} .$
\end{enumerate}

The computational complexity of the algorithm is $O\left(s^2
n\right)$, where $s$ is the sample size and $n$ is the number of
rows and columns of $M$.  The bottleneck is in the computation of
the matrix product $Z^TZ$.

\subsubsection{ A Single-Step Solution for the SVD of $\hat{\mathbf{M}}$ }
\label{sec:symmetric_svd_oneshot_sol}

The ``one-shot'' solution in \cite{Chung} assumes that $A_M$ has a
square root matrix $A^{1/2}_M$. This assumption is true if the
matrix is positive definite. Otherwise, it imposes some limitations
on $A_M$. These will be discussed in section $\ref{sec:single_step_prereq}$.

Let $A^{-1/2}_M$ be the pseudo-inverse of the square root matrix of
$A_M$. Denote $G^T=A^{-1/2}_M\left[ \begin{array}{cc} A_M & B_M
\end{array}\right]$. From this definition we have $\hat{M}=GG^T$.
The matrix $S\in {\mathbb R}^{s\times s}$ was defined in
\cite{Chung}, where $S=G^TG=A_M+A^{-1/2}_MB_MB^T_MA^{-1/2}_M$. $S$
is fully decomposed as $U_S{\Lambda }_SU^T_S$. The orthogonal
eigenvectors of $\hat{M}$ are formed as $U_o=GU_S{\Lambda
}^{-1/2}_S$ and the eigenvalues are given in ${\Lambda }_S$.

The following required identities, as in section
\ref{sec:symmetric_svd_gen sol}, can again be verified as follows:
\begin{enumerate}
\item  Bi-orthogonality: \\ $U^T_oU_o={\Lambda }^{-1/2}_SU^T_SG^TGU_S{\Lambda }^{-1/2}_S={\Lambda }^{-1/2}_SU^T_SSU_S{\Lambda }^{-1/2}_S={\Lambda }^{-1/2}_SU^T_S\cdot U_S{\Lambda }_SU^T_S\cdot U_S{\Lambda }^{-1/2}_S=I.$
\item  SVD form: $U_o{\Lambda }_SU^T_o=GU_S{\Lambda }^{-1/2}_S\cdot {\Lambda }_S\cdot {\Lambda}^{-1/2}_SU^T_SG^T=GG^T=\hat{M}.$
\end{enumerate}

The computational complexity remains the same (the bottleneck of the
algorithm is the formation of $B_MB^T_M$). However this version is
numerically more accurate. According to \cite{Chung}, the extra
calculations in the general method of solution lead to an increase
in the loss of significant digits.

\section{ Nystr\"{o}m-like SVD approximation }
\label{sec:nystrom_like_svd}

The SVD of a matrix can also be approximated via the basic
quadrature technique of the Nystr\"{o}m method. In this case, we do
not require an eigen-decomposition. Therefore,  $M$ does not
necessarily have to be square. Let $M\in{\mathbb R}^{m\times n}$ be
a matrix with the decomposition given in Eq.
\eqref{eq:assym_m_decomposition}. We begin with the SVD form
$A_M=U\Lambda H$ where $U,H\in {\mathbb R}^{s\times s}$ are unitary
matrices and ${\Lambda}\in {\mathbb R}^{s\times s}$ is diagonal. We
assume that zero is not a singular value of $A_M$. Accordingly, $U$
can be formulated as:
\begin{equation} \label{part_svd_form}
U = A_M H {\Lambda}^{-1}.
\end{equation}

Let $u^i,h^i\in {\mathbb R}^s$ be the $i^{th}$ columns in $U$ and
$H$, respectively. Let $u^i=\{ u^i_l \}_{l=1}^{s}$ be the partition
of $u^i$ into elements. By using Eq. \eqref{part_svd_form}, each
element $u^i_l$ can be presented as the sum $u^i_l =
\frac{1}{\lambda_i}\sum_{j=1}^{n} { M_{lj}\cdot h^i_j }$.

We can use the entries of $F_M$ as interpolation weights for
extending the singular vector $u^i$ to the $k^{th}$ row of $M$,
where $s+1 \leq k \leq n$. Let $\tilde{u}^i = \{ \tilde{u}^i_{k-s}
\}_{k=s+1}^{n} \in {\mathbb R}^{n-s}$ be a column vector that
contains all the approximated entries. Each element
$\tilde{u}^i_{k-s}$ will be calculated as $\tilde{u}^i_{k-s} =
\frac{1}{\lambda_i}\sum_{j=1}^{n} { M_{kj}\cdot h^i_j }$. Therefore,
the matrix form of $\tilde{u}^i$ becomes $\tilde{u}^i =
\frac{1}{\lambda_i} F_M\cdot h^i$.

Putting together all the $\tilde{u}^i$'s as $\tilde{U} =
\left[\begin{array}{cccc} \tilde{u}^1 & \tilde{u}^2 & \dots &
\tilde{u}^s\end{array}\right] \in {\mathbb R}^{n-s\times s}$, we get
$\tilde{U}=F_MH{\Lambda }^{-1}$.

The basic SVD equation of $A_M$ can also be written as $H=A_M^T U
\Lambda^{-1}$. We approximate the right singular vectors of the
out-of-sample columns by employing a symmetric argument. We obtain
$\tilde{H}=B^T_MU{\Lambda }^{-1}$.

The full approximations of the left and right singular vectors of
$\hat{M}$, denoted by $\hat{U}$ and $\hat{H}$, respectively, are
\begin{equation}\label{eq:nystrom_svd_init_form}
\hat{U}=\left[ \begin{array}{c}
U \\
F_MH{\Lambda }^{-1} \end{array} \right],\ \ \hat{H}=\left[
\begin{array}{c}
H \\
B^T_MU{\Lambda }^{-1} \end{array} \right].
\end{equation}

The explicit ``Nystr\"{o}m'' form of $\hat{M}$ becomes
\begin{equation}\label{eq:m_hat_general_form}
\begin{array}{c}
\hat{M}=\hat{U}\Lambda {\hat{H}}^T=\left[ \begin{array}{c}
U \\
F_MH{\Lambda }^{-1} \end{array} \right]\Lambda \left[
\begin{array}{cc} H^T & {\Lambda }^{-1}U^TB_M \end{array}
\right]=\left[ \begin{array}{cc}
A_M & B_M \\
F_M & F_MA^+_MB_M \end{array}
\right]= \\
\left[ \begin{array}{c}
A_M \\
F_M \end{array} \right]A^+_M\left[ \begin{array}{cc} A_M & B_M
\end{array} \right]
\end{array}
\end{equation}
where $A^+_M$ denotes the pseudo-inverse of $A_M$. $\hat{M}$ does
not modify $A_M, B_M$ and $F_M$ but approximates $C_M$ by $F_M A_M^+
B_M$. Note that the Nystr\"{o}m matrix form of the SVD is similar to
Eq. \eqref{eq:assym_matrix_nys_form}, which is the Nystr\"{o}m form
of the EVD matrix.

\section{ Decomposition of General Matrices }
\label{sec:mat_gen_decomp}

We will refer to a decomposition of $M$ given in Eq.
\eqref{eq:assym_m_decomposition} with the corresponding
decomposition into $A_M,B_M,F_M$ and $C_M$. $\hat{M}$ denotes the
approximated Nystr\"{o}m matrix.

This section presents procedures for explicit orthogonalization of
the singular-vectors and eigenvectors of $\hat{M}$. Starting with
$\hat{M}$ in the form of Eqs. \eqref{eq:assym_matrix_nys_form} and
\eqref{eq:m_hat_general_form}, we find its canonical SVD and EVD
form, respectively. Constructing these representations takes time
and space that are linear in the dimensions of $M$.

\subsection{ Construction of EVD for $\hat{\textbf{M}}$ }
\label{sec:evd_general_sol}

Let $M$ be a square matrix. We will approximate the eigenvalue
decomposition of $\hat{M}$ without explicitly forming $\hat{M}$.

We begin with a matrix $M$ that is partitioned as in Eq.
\eqref{eq:assym_m_decomposition}. By explicitly employing the
Nystr\"{o}m method, we construct $\hat{U}$ and $\hat{V}$ as defined
in Eq. \eqref{eq:assym_ext_vec}. Then, we proceed by defining the
matrices $G_U=\hat{U}{\Lambda }^{1/2}$ and $G_V={\Lambda
}^{1/2}\hat{V}$. We directly compute the EVD of $G_VG_U$ as $F\Sigma
F^{-1}$. The eigenvalues of $\hat{M}$ are given by $\Sigma $ and the
right and left eigenvectors are $U_o=G_UF{\Sigma }^{-1/2}$ and
$V_o={\Sigma }^{-1/2}F^{-1}G_V$, respectively.

The left and right eigenvectors are mutually orthogonal since
\[V_oU_o={\Sigma }^{-1/2}F^{-1}G_V\cdot G_UF{\Sigma }^{-1/2}={\Sigma }^{-1/2}F^{-1}\cdot F\Sigma F^{-1}\cdot F{\Sigma }^{-1/2}=I. \]
The EVD form of $U_o, V_o$ and $\Sigma$ gives $\hat{M}$, as we see
from
\[U_o\Sigma V_o=G_UF{\Sigma }^{-1/2}\cdot \Sigma \cdot {\Sigma }^{-1/2}F^{-1}G_V=G_UG_V=\hat{U}{\Lambda }^{1/2}\cdot {\Lambda }^{1/2}\hat{V}=\hat{M}. \]
These two properties qualify $U_o\Sigma V_o$ as the EVD of
$\hat{M}$.

When $M$ is symmetric, the matrix $G_V$ is simply $G^T_U$. By using
the terminology in section \ref{sec:symmetric_svd_gen sol}, we
denote $G_V=Z$ and the matrix $G_VG_U$ is transformed into $ZZ^T$.
From here on the method of solution in section
\ref{sec:symmetric_svd_gen sol} coincides with the current section.
Hence, this form of EVD approximation generalizes the symmetric
case.

The computational complexity is $O(s^2n)$, where $s$ is the sample
size (the size of $A_M$) and $n$ is the size of $M$. The
computational bottleneck is in the formation of $G_VG_U$.

\subsubsection{ A Single-Step Solution for the EVD for $\hat{\mathbf{M}}$ }
\label{sec:evd_oneshot_sol} This solution method assumes that $A_M$
has a square root matrix $A^{1/2}_M$. From this assumption, we can
modify the algorithm in section \ref{sec:evd_general_sol} to
construct the EVD of $\hat{M}$ with fewer steps.

We define the matrices $G_U$ and $G_V$ to be
\[G_U=\left[ \begin{array}{c}
A_M \\
F_M \end{array} \right]A^{-1/2}_M,\ \ \ G_V=A^{-1/2}_M\left[
\begin{array}{cc} A_M & B_M \end{array} \right] .\] We proceed to
explicitly compute the eigen-decomposition of $G_VG_U\in {\mathbb
R}^{s\times s}$ as $G_VG_U=F\Sigma F^{-1}$. The eigenvalues of
$\hat{M}$ are given by $\Sigma $ and the right and left eigenvectors
of $\hat{M}$ are formed by $U_o=G_UF{\Sigma }^{-1/2}$ and
$V_o={\Sigma }^{-1/2}F^{-1}G_V$, respectively. Again, we can verify
the eigenvectors are mutually orthogonal:
\[V_oU_o={\Sigma }^{-1/2}F^{-1}G_V\cdot G_UF{\Sigma }^{-1/2}={\Sigma }^{-1/2}F^{-1}\cdot G_VG_U\cdot F{\Sigma }^{-1/2}={\Sigma }^{-1/2}F^{-1}\cdot F\Sigma F^{-1}\cdot F{\Sigma }^{-1/2}=I,\]
and the matrices $U_o, V_o$ and ${\Lambda}_S$ form $\hat{M}$ as
\[U_o{\Lambda }_SV_o=G_UF{\Sigma }^{-1/2}\cdot \Sigma \cdot {\Sigma }^{-1/2}F^{-1}G_V=G_UG_V=\left[ \begin{array}{c}
A_M \\
F_M \end{array} \right]A^{-1/2}_M\cdot A^{-1/2}_M\left[
\begin{array}{cc} A_M & B_M \end{array} \right]=\hat{M}.\]
The reduction to the symmetric case is straightforward here as well.
We have $G_V=G^T_U$ when $M$ is symmetric. By using the terms of
section \ref{sec:symmetric_svd_oneshot_sol}, we have $G^T_U=G_V=G$.
The expression $G_VG_U$ turns into $G^TG$. After that point the
methods of solution coincide.

Again, the algorithm takes $O(s^2n)$ operations due to the need to
calculate $G_VG_U$. Compared to the solution given in section
\ref{sec:evd_general_sol}, the single-step solution performs fewer
matrix operations. Therefore, it achieves better numerical accuracy.

\subsection{ Construction of SVD for $\hat{\textbf{M}}$ }
\label{sec:svd_general_sol}

Let $M$ be a general $m\times n$ matrix with the decomposition in
Eq. \eqref{eq:assym_m_decomposition}. Given an initial sample $A_M$,
we present an algorithm that efficiently computes the SVD of
$\hat{M}$ (defined by Eq. \eqref{eq:assym_matrix_nys_form}).

We explicitly compute the SVD of $A_M$ and use the technique
outlined in section \ref{sec:nystrom_like_svd} to obtain $\hat{U}$
and $\hat{H}$ as in Eq. \eqref{eq:nystrom_svd_init_form}. We form
the matrices $Z_U=\hat{U}{\Lambda }^{1/2}$ and $Z_H=\hat{H}{\Lambda
}^{{\rm 1/2}}$. We proceed by forming the symmetric $s\times s$
matrices $Z^T_UZ_U$ and $Z^T_HZ_H$ and compute their SVD as
$Z^T_UZ_U=F_U{\Sigma }_UF^T_U$ and $Z^T_HZ_H=F_H{\Sigma }_HF^T_H$,
respectively. The next stage derives an SVD form for the $s\times s$
matrix $D={\Sigma }^{1/2}_UF^T_UF_H{\Sigma }^{1/2}_H$. This is given
explicitly by computing $D=U_D{\Lambda }_DH^T_D$. The singular
values of $\hat{M}$ are given in ${\Lambda }_D$ and the leading left
and right singular vectors of $\hat{M}$ are $U_o=Z_UF_U{\Sigma
}^{-1/2}_UU_D$ and $H_o=Z_HF_H{\Sigma }^{-1/2}_HH_D$, respectively.
The columns of $U_o$ and $H_o$ are orthogonal since
\[U^T_oU_o=U^T_D{\Sigma }^{-1/2}_UF^T_UZ^T_U\cdot Z_UF_U{\Sigma }^{-1/2}_UU_D=U^T_D{\Sigma }^{-1/2}_UF^T_U\cdot F_U{\Sigma }_UF^T_U\cdot F_U{\Sigma }^{-1/2}_UU_D=U^T_DU_D=I,\]
\[H^T_oH_o=H^T_D{\Sigma }^{-1/2}_HF^T_HZ^T_H\cdot Z_HF_H{\Sigma }^{-1/2}_HH_D=H^T_D{\Sigma }^{-1/2}_HF^T_H\cdot F_H{\Sigma }_HF^T_H\cdot F_H{\Sigma }^{-1/2}_HH_D=H^T_DH_D=I.\]
The SVD of $\hat{M}$ is formed by using $U_o, H_o$ and $V_D$
\[\begin{array}{c}
U_o{\Lambda_D }_oH^T_o=Z_UF_U{\Sigma }^{-1/2}_UU_D\cdot {\Lambda }_D\cdot H^T_D{\Sigma }^{-1/2}_HF^T_HZ^T_H=Z_UF_U{\Sigma }^{-1/2}_U\cdot D\cdot {\Sigma }^{-1/2}_HF^T_HZ^T_H= \\
= Z_UF_U{\Sigma }^{-1/2}_U\cdot {\Sigma }^{1/2}_UF^T_UF_H{\Sigma
}^{1/2}_H\cdot {\Sigma }^{-1/2}_HF^T_HZ^T_H=Z_UZ^T_H=\hat{U}{\Lambda
}^{1/2}\cdot {\Lambda }^{1/2}{\hat{H}}^T=\hat{M}.
\end{array}\]

When $M$ is symmetric, this solution method coincides with the
method in section \ref{sec:symmetric_svd_gen sol}. The matrices
$Z_U$ and $Z_H$ correspond to $Z$ in section
\ref{sec:symmetric_svd_gen sol}. The matrix $D$ becomes the diagonal
matrix $\Sigma$ of the symmetric case. The computational complexity
of the procedure is $O\left(s^2\left(m+n\right)\right)$. The
bottleneck is the computation of $Z^T_UZ_U$ and $Z^T_HZ_H$.

\subsubsection{  A Single-Step Solution for the SVD of $\hat{\mathbf{M}}$ }
\label{sec:svd_oneshot_sol} This solution method assumes that $A_M$
has a square root matrix $A^{1/2}_M$. Similar to section
\ref{sec:evd_oneshot_sol}, this assumption allows us to modify the
algorithm of the general case to achieve the same result in fewer
steps.

Let $A^{-1/2}_M$ be the pseudo-inverse of the square root matrix of
$A_M$. We begin by forming the matrices $G_U$ and $G_H$ such that
\[G_U=\left[ \begin{array}{c}
A_M \\
F_M \end{array} \right]A^{-1/2}_M,\ \ G_H={\left(A^{-1/2}_M\left[
\begin{array}{cc} A_M & B_M \end{array} \right]\right)}^T . \] The
symmetric matrices $G^T_UG_U$ and $G^T_HG_H$ are diagonalized by
$G^T_UG_U=F_U{\Sigma }_UF^T_U$ and  $G^T_HG_H=F_H{\Sigma }_HF^T_H$.
From these parts we form $D={\Sigma }^{1/2}_UF^T_UF_H{\Sigma
}^{1/2}_H$ which is explicitly diagonalized as $D=U_D{\Lambda
}_DH^T_D$. The singular values of $\hat{M}$ are given by ${\Lambda
}_D$ and the left and right singular vectors are given by
$U_o=G_UF_U{\Sigma }^{-1/2}_UU_D$ and $H_o=G_HF_H{\Sigma
}^{-1/2}_HH_D$, respectively.

As in section \ref{sec:svd_general_sol}, we can verify the
identities that make this decomposition a valid SVD. The singular
vectors are orthogonal:
\[U^T_oU_o=U^T_D{\Sigma }^{-1/2}_UF^T_UG^T_U\cdot G_UF_U{\Sigma }^{-1/2}_UU_D=U^T_D{\Sigma }^{-1/2}_UF^T_U\cdot F_U{\Sigma }_UF^T_U\cdot F_U{\Sigma }^{-1/2}_UU_D=U^T_DU_D=I,\]
\[H^T_oH_o=H^T_D{\Sigma }^{-1/2}_HF^T_HG^T_H\cdot G_HF_H{\Sigma }^{-1/2}_HH_D=H^T_D{\Sigma }^{-1/2}_HF^T_H\cdot F_H{\Sigma }_HF^T_H\cdot F_H{\Sigma }^{-1/2}_HH_D=H^T_DH_D=I.\]
The SVD is formed by $U_o, H_o$ and ${\Lambda}_D$:
\[
\begin{array}{cl}
U_o{\Lambda }_DH^T_o=G_UF_U{\Sigma }^{-1/2}_UU_D\cdot {\Lambda }_D\cdot H^T_D{\Sigma }^{-1/2}_HF^T_HG^T_H=G_UF_U{\Sigma }^{-1/2}_U\cdot D\cdot {\Sigma }^{-1/2}_HF^T_HG^T_H= \\
=G_UF_U{\Sigma }^{-1/2}_U\cdot {\Sigma }^{1/2}_UF^T_UF_H{\Sigma
}^{1/2}_H\cdot {\Sigma }^{-1/2}_HF^T_HG^T_H=G_UG^T_H=\left[
\begin{array}{c}
A_M \\
F_M \end{array} \right]A^{-1/2}_M\cdot A^{-1/2}_M\left[
\begin{array}{cc} A_M & B_M \end{array} \right]=\hat{M}.
\end{array}
\]

If $M$ is symmetric, this method reduces to the single-step solution
described in section \ref{sec:symmetric_svd_oneshot_sol}. The
matrices $G_U$ and $G_H$ correspond to $G$ in the symmetric case.
The matrix $D$ becomes ${\Lambda }_S$.

The computational complexity of the procedure remains
$O\left(s^2\left(m+n\right)\right)$. The computational bottleneck of
the algorithm is in the formation of $G^T_UG_U$.

\subsection{ Prerequisite for the Single-Step method }
\label{sec:single_step_prereq}

The single-step methods, described in sections
\ref{sec:symmetric_svd_oneshot_sol}, \ref{sec:evd_oneshot_sol} and
\ref{sec:svd_oneshot_sol}, require that $A_M$ have a square root
matrix.

When a matrix is positive semi-definite, a square root can be found
via the Cholesky factorization algorithm (\cite{MatrixComputations}
chapter 4.2.3). But positive-definiteness is not a necessary
prerequisite. For example, the square root of a diagonalizable
matrix can be found via its diagonalization. If $A_M=U\Lambda
U^{-1}$, then, $A^{1/2}_M=U{\Lambda }^{1/2}U^{-1}$. In this case,
the matrix does not need to be invertible.

It can be shown that under a complex realm, every non-singular
matrix has a square root. An algorithm for calculating the square
root for a given non-singular matrix is given in \cite{MatrixSqrt}.
This suggests a way of assuring the existence of a square root
matrix. We can make $A_M$ non-singular, or equivalently, a full rank
matrix.

The rank of $A_M$ will also have a role in bounding the
approximation error of the Nystr\"{o}m procedure. This will be
elaborated in section \ref{sec:nystrom_accuracy}.

\section{ Choice of Sub-Sample }
\label{sec:sample_choice}

The choice of initial sample for performing the Nystr\"{o}m
extension is an important part in the approximation procedure. The
sample matrix $A_M$ is determined by permutation of the rows and
columns of $M$ (as given in Eq. \eqref{eq:assym_m_decomposition}).
Our goal is to choose a (possibly constrained) permutation of $M$
such that the resulting matrix can be approximated more accurately
by the Nystr\"{o}m method. Here accuracy is measured by $L_2$
distance between the pivoted version of $M$ and the Nystr\"{o}m
approximated version. This notion is made precise in section
\ref{sec:nystrom_accuracy}.

We allow for complete pivoting in the choice of a permutation for
$M$. This means that both columns and rows can be independently
permuted. This kind of pivoting does not generally preserve the
eigenvalues and eigenvectors of the matrix. However, the singular
values of the matrix remain unchanged and the singular vectors are
permuted. Formally, let $E_r$ and $E_c$ be the row and column
permutation matrices, respectively. Using the SVD of $M$, the
pivoted matrix is decomposed as $E_r M E_c = E_r U \Sigma V^T E_c =
\left(E_r U\right) \Sigma \left(V^T E_c\right)$. Row and column
permutations leave $U$ and $V^T$ unitary. Therefore $\left(E_r
U\right) \Sigma \left(V^T E_c\right)$ is the SVD of $E_r M E_c$. The
singular vectors of $M$ can be easily regenerated by permuting the
left and right singular vectors of $E_r M E_c$ by $E_r^{-1}$ and
$E_c^{-1}$ respectively.

Section \ref{sec:nystrom_accuracy} shows the choice of $A_M$
determines the Nystr\"{o}m approximation error. Hence, the problem
of choosing a sample is equivalent to choosing the rows and columns
of $M$ whose intersection forms $A_M$. Therefore, it makes sense to
use the size $s$ of $A_M$ as our sample size. This size largely
determines the time and space complexity of the presented
approximation procedures. The complexities are
$O\left(s^2\left(m+n\right)\right)$ and
$O\left(s\left(m+n\right)\right)$, respectively.

\subsection{ Related Work on Sub-Sample Selection }
\label{sec:gram_matrices_related_work} Previous works on sub-sample
selection focused on kernel matrices. These were done for symmetric
matrices where the entries represent affinities. In these settings, we
can use a single permutation for the columns and rows without
changing the original meaning of the matrix. This pivoting variant
is called symmetric pivoting. Sample selection algorithms for kernel
matrices try to find a permutation matrix $E_p$ such that $E_p^T M
E_p$ is most accurately approximated by the Nystr\"{o}m method.

The simplest sample selection method is based on random sampling. It
works well for dense image data (\cite{Chung}). Random sampling is
also used in \cite{SamplePcaGreedy} while employing a greedy
criterion that helps to determine the quality of the sample. A
different greedy approach for sample selection is used in
\cite{SampleGreedy}, where a new point is added to the sample based
on its distance from a constrained linear combination of previously
selected points.

In \cite{SampleKmeans}, the \emph{k}-means clustering algorithm is
used for selecting the sub-sample. The \emph{k}-means cluster
centers are shown to minimize an error criterion related to the
Nystr\"{o}m approximation error. Finally,  Incomplete Cholesky
Decomposition (ICD) (\cite{ICD}) employs the pivoted Choleksy
algorithm and uses a greedy stopping criterion to determine the
required sample size for a given approximation accuracy.

The Cholesky decomposition of a matrix factors it into $Z^T Z$,
where $Z$ is an upper triangular matrix. Initially, $Z = 0$. The ICD
algorithm applies the Cholesky decomposition to $M$ while
symmetrically pivoting the columns and rows of $M$ according to a
greedy criterion. The algorithm has an outer loop  that scans the
columns of $M$ according to a pivoting order. The results for each
column determine the next column to scan. This loop is terminated
early after $s$ columns were scanned by using a heuristic on the
trace of the residual $Z^T Z - M$. This algorithm (\cite{ICD})
approximates $M$. This is equivalent to a Nystr\"{o}m approximation
where the initial sample is taken as the intersection of the pivoted
columns and rows.

When $M$ is a Gram matrix, it can be expressed as the product of two
matrices. Let $M$ be decomposed into $M=X^T X$ where $X\in{\mathbb
R}^{n\times n}$. The special properties of $M$ were exploited
differently in \cite{NystromDrineas}. Specifically, the fact that
$M_{ii}$ is the norm of the column $X_i$ is used. A non-Gram matrix
requires $O(n^2)$ additional operations to compute $X_i^T X_i$,
which is impractical for large matrices. Once the norms of the
columns in $X$ are known, a method similar to \cite{LinearTimeSvd}
is used to choose a good column sample from $X$. The intersection in
$M$ of the pivoted columns and the corresponding rows is a good
choice for $A_M$. The Nystr\"{o}m procedure is then performed
similarly to what was described in  section
\ref{sec:symmetric_svd_gen sol}. The runtime complexity of the
algorithm in \cite{NystromDrineas} is $O(n)$.

\subsection{ Preliminaries }

\begin{definition}\textbf{Approximate `thin' Matrix Decomposition.} Given a matrix $M\in {\mathbb R}^{m\times n}$. A "thin" matrix decomposition is an approximation of the form $M=GS$ where $G\in {\mathbb R}^{m\times k}$, $S\in {\mathbb R}^{k\times n}$ and $k\leq min(m,n)$.
\end{definition}
This form effectively approximates $M$ using a rank-$k$ matrix
product. A good example for such an approximation is the truncated
rank-$k$ SVD. It approximates a $m\times n$ matrix as $U\Lambda
V^T$, where $U\in {\mathbb R}^{m\times k}, \Lambda\in {\mathbb
R}^{k\times k}$ and $V\in{\mathbb R}^{n\times k}$. When this
decomposition is employed, we can choose, for example, $G=U,
S=\Lambda V^T$. Many algorithms
(\cite{LinearTimeSvd,SVDDeshpande,SVDHarPeled,SVDSarlos}) exist for
approximating the rank-$k$ SVD with a runtime close to $O(mn)$.

Truncated SVD is a popular choice, but it is by no means the only
one. Other examples include truncated pivoted QR
(\cite{PivotedTruncQr}) or the interpolative decomposition (ID) as
outlined in \cite{ID}.

\begin{definition}\textbf{Numerical Rank.}\label{def:numer_rank} A matrix $A$ has numerical rank $r$  with respect to a threshold $\epsilon$ if ${\sigma }_{r+1}(A)$ is the first singular value such that
\[\frac{{\sigma }_1\left(A\right)}{{\sigma }_{r+1}(A)}>\epsilon .\]
\end{definition}
This definition generalizes the $L_2$ condition number (${\kappa
}_2\left(A\right)$), since it also applies to non-invertible and
non-square matrices.

\begin{definition}\textbf{ Rank Revealing ${\mathbf {QR}}$ Decomposition (RRQR).}
Let $A\in {\mathbb R}^{m\times n}$ be a matrix and let $k$ be a user
defined threshold. A RRQR algorithm finds a permutation matrix $E$
such that $AE$ has a $QR$ decomposition with special properties.
Formally, we write $AE=QR$ such that $Q$ is an orthogonal matrix and
$R$ is upper triangular. Let $R$ have the following decomposition:
\begin{equation}\label{eq:r_from_qr}
R=\left[ \begin{array}{cc}
R_{11} & R_{12} \\
0 & R_{22} \end{array} \right]
\end{equation}
where $R_{11}\in {{\mathbb R}}^{k\times k},R_{12}\in {\mathbb
R}^{k\times\left(n-k\right)}$ and $R_{22}\in {\mathbb
R}^{\left(m-k\right)\times\left(n-k\right)}$. Let
$p\left(k,n\right)$ be a fixed non-negative function bounded by a
low degree polynomial in $k$ and $n$. A RRQR algorithm tries to
permute the columns of $A$ such that
\[
{\sigma }_k\left(R_{11}\right)\ge \frac{{\sigma
}_k\left(A\right)}{p(k,n)},\ \ \ {\sigma }_1\left(R_{22}\right)\le
{\sigma}_{k+1}\left(A\right)\cdot p(k,n) .
\]
\end{definition}
An overview on this topic is given in \cite{StrongRRQR}.

The relation between $A$ and $R$ can shed some light on the
rank-revealing properties of RRQR. Let $AE=\left[
\begin{array}{cc}A_1 & A_2 \end{array}\right]$ be a partitioning of
$AE$ such that $A_1$ contains the first $k$ columns. The RRQR
decomposition is rank-revealing in the sense that it tries to put a
set of $k$ maximally independent columns of $A$ into $A_1$. We
formalize this statement with Lemma \ref{lem:rrqr_accuracy}.

\begin{lemma}
\label{lem:rrqr_accuracy} Assume that the RRQR algorithm found a
pivoting of $A$ such that ${\sigma }_k\left(R_{11}\right)\ge
{{\sigma }_k\left(A\right)}/{\beta }$, where $\beta \geq 1$. If $A$
has numerical rank of at least $k$ with respect to the threshold
$\epsilon$, then, the numerical rank of $A_1$ (the first $k$ columns
of $AE$) is $k$ with respect to the threshold $\beta \cdot \epsilon
$.
\end{lemma}

\begin{proof}
The RRQR algorithm yields $A_1=Q\left[ \begin{array}{cc}R_{11} & 0
\end{array}\right]^T $. Since $Q$ is orthogonal, it does not modify
singular values. Therefore, we have ${\sigma
}_k\left(A_1\right)={\sigma }_k\left[ \begin{array}{cc} R_{11} & 0
\end{array} \right]^T={\sigma }_k\left(R_{11}\right)$. By combining
the above with our assumption on the RRQR algorithm, we get
\begin{equation}\label{eq:qrguarantee}
\beta \cdot {\sigma }_k\left(A_1\right)\ge {\sigma }_k\left(A\right)
.
\end{equation}
The interlacing property of singular values (Corollary 8.6.3 in
\cite{MatrixComputations}) gives us
\begin{equation}\label{eq:singinterlace}
{\sigma }_1\left(A\right)\ge {\sigma }_1\left(A_1\right) .
\end{equation}
By employing definition \ref{def:numer_rank} for $A$ and
incorporating Eqs. \eqref{eq:qrguarantee} and
\eqref{eq:singinterlace}, we get
\[\epsilon\ge \frac{{\sigma }_1\left(A\right)}{{\sigma }_k(A)}\ge \frac{{\sigma }_1\left(A_1\right)}{{\sigma }_k(A)}\ge \frac{{\sigma }_1\left(A_1\right)}{\beta \cdot {\sigma }_k\left(A_1\right)} .\]
By rearranging terms, we get
\[\frac{{\sigma }_1\left(A_1\right)}{{\sigma }_k\left(A_1\right)}\le \beta \cdot \epsilon .\]
Therefore the numerical rank of $A_1$ is at least $k$ with respect
to the threshold $\beta \cdot \epsilon$. Since $A_1$ has only $k$
columns, it has precisely this rank.
\end{proof}

\subsection{ Algorithm Description and Rationale}

Initially, our algorithm decomposes the matrix $M$ into $G\cdot S$.
Then, a RRQR algorithm chooses the $s$ \emph{most} non-singular
columns of $G^T$ and $S$ and insert then into $G_A^T$ and $S_A$,
respectively. We use a variant of RRQR that measures non-singularity
according to the magnitude of the last singular value (see the proof
of Corollary \ref{cor:algo_rrqr_bound}). The non-singularity of
$G_A$ and $S_A$ will bound the non-singularity of $G_A S_A$ (see Eq.
\eqref{eq:matinv_approx}).

On a higher level observation, the algorithm will try to perform an
exhaustive search for the $s\times s$ most non-singular square in
$GS$. However, since $GS$ approximates $M$, choosing $A_M$ from the
same rows and columns of $M$ amounts to choosing one of its most
non-singular squares. These notions are formalized in Theorem
\ref{th:a_is_invertible}.

The magnitude of the last singular-value in $A_M$, denoted by
$\sigma_s\left(A_M\right)$, will be used as a measure for the
singularity of $A_M$. This quantity is instrumental in defining  the
bound of the approximation error given in Theorem
\ref{th:algo_approx_error}. We show in the experimental results
section (section \ref{sec:experiments}) that empirically,
$\sigma_s\left(A_M\right)$ is strongly related to the approximation
error of the Nystr\"{o}m procedure.

\subsection{ Analysis of Nystr\"{o}m Error }
\label{sec:nystrom_accuracy}

Let $M$ be a matrix with the decomposition given by Eq.
\eqref{eq:assym_m_decomposition}. This partitioning corresponds to
sampling $s$ columns and rows from $M$ to form the matrix $A_M$. Our
error analysis depends on an approximate decomposition of $M$ into a
product of two `thin' matrices. Let $M\simeq GS$ be a decomposition
of $M$ where $G\in {\mathbb R}^{m\times s}$ and $S\in {\mathbb
R}^{s\times n}$. The approximation error of $M$ by $GS$ is denoted
by $e_s$. Formally, $\vectornorm{M-GS}_2 \leq e_s$. Let $G=\left[
\begin{array}{cc}G_A & G_B \end{array}\right]^T$ be a row
partitioning of $G$ where $G_A\in {\mathbb R}^{s\times r}$ and
$G_B\in {\mathbb R}^{\left(m-s\right)\times r}$. Let
$S=[\begin{array}{cc}S_A & S_B \end{array}]$ be a column
partitioning of $S$ where $S_A\in {\mathbb R}^{r\times s},S_B\in
{\mathbb R}^{r\times \left(n-s\right)}$. This notation yields the
following forms for the sub-matrices of $M$:
\begin{equation}\label{eq:nyst_rank_s_decomposition}
A_M\simeq G_AS_A,\ \ B_M\simeq G_AS_B,\ \ F_M\simeq G_BS_A,\ \
C_M\simeq G_BS_B.
\end{equation} where $A_M$, $B_M$, $F_M$ and $C_M$ were defined in Eq. \ref{eq:assym_m_decomposition}.

\begin{lemma}\textbf{(based on Corollary 8.6.2 in \cite{MatrixComputations})}\label{lem:sing_val_diff}
If $A$ and $A+E$ are in $\mathbb{R}^{m\times n}$ then for $k\leq
min\left(m,n\right)$ we have $\left|\sigma_k\left(A+E\right) -
\sigma_k\left(A\right)\right|\leq\sigma_1\left(E\right) =
\vectornorm{E}_2$.
\end{lemma}
\begin{proof}
Corollary 8.6.2 in \cite{MatrixComputations} states the same lemma
with the requirement $m\geq n$. If $m<n$, we can use the original
version of the lemma to get $\left|\sigma_k\left(A^T+E^T\right) -
\sigma_k\left(A^T\right)\right|\leq\vectornorm{E^T}_2$.
Transposition neither modifies the singular values nor the norm of a
matrix.
\end{proof}

\begin{theorem}\label{th:a_is_invertible}
Assuming that
\begin{enumerate}

\item \label{gs_nonsing} $\sigma_s\left(GS\right) > 0$;

(This means that $GS$ is of rank at least $s$. Otherwise, a
non-singular $A_M$ cannot be found)

\item \label{svd_assump} $\sigma_s\left(G\right) \sigma_s\left(S\right) = \sigma_s\left(GS\right) / \gamma$ for some constant $\gamma \geq 1 ;$

(It will allow us to use the non-singularity of $G_A$ and $S_A$ as a
bound for the non-singularity of $G_A S_A$. This demands the initial
decomposition to be reasonably well conditioned. See Corollary
\ref{cor:algo_svd_bound} for details)

\item \label{rrqr_assump} $\sigma_s(G_A) \geq \sigma_s(G) / \beta$ and $\sigma_s(S_A) \geq \sigma_s(S) / \beta$ for some constant  $\beta \geq 1;$

(This will allow us to use $\sigma_s\left(A_M\right)$ as a bound for
$\sigma_s\left(G S\right)$. The RRQR algorithm will fulfill this
assumption in its choice of $G_A$ and $S_A$)

\item \label{singval_assump} $e_s < \left(\sigma_s\left(M\right)-e_s\right) / \beta^2 \gamma$, where $e_s$ is the error given by the rank-$s$ approximation of $M$ by $GS$.

(The initial rank-$s$ approximation should be good enough)

\end{enumerate}
Then, $A_M$ is non-singular.
\end{theorem}

\begin{proof}
Lemma \ref{lem:sing_val_diff} yields $\abs{\sigma_s\left(M\right) -
\sigma_s\left(GS\right)} \leq \vectornorm{M - GS}_2 = e_s$, or
\begin{equation}\label{eq:gs_approx}
\sigma_s\left(M\right) - e_s \leq \sigma_s\left(GS\right) .
\end{equation}
From assumptions \ref{svd_assump} and \ref{rrqr_assump} we obtain
\begin{equation}\label{eq:gasa_approx}
\sigma_s\left(GS\right) / \beta^2 \gamma \leq \sigma_s\left(G\right)
\sigma_s\left(S\right) / \beta^2 \leq \sigma_s\left(G_A\right)
\sigma_s\left(S_A\right) .
\end{equation}
$G_A$ and $S_A$ are $s\times s$ matrices. Assumptions
\ref{gs_nonsing}, \ref{svd_assump} and \ref{rrqr_assump} show that
$\sigma_s\left(G_A\right)$ and $\sigma_s\left(G_A\right)$ are
non-zero. Thus, $G_A$ and $S_A$ are non-singular and we obtain
\begin{equation}\label{eq:matinv_approx}
\sigma_s\left(G_A\right) \sigma_s\left(S_A\right) =
\frac{1}{\vectornorm{G_A^{-1}}\vectornorm{S_A^{-1}}} \leq
\frac{1}{\vectornorm{S_A^{-1} G_A^{-1}}} =
\frac{1}{\vectornorm{\left(G_A S_A\right)^{-1}}} = \sigma_s\left(G_A
S_A\right) .
\end{equation}
By combining Eqs. \eqref{eq:gs_approx}, \eqref{eq:gasa_approx} and
\eqref{eq:matinv_approx} we get
\begin{equation}\label{eq:nonsing_left_side}
\left(\sigma_s\left(M\right)-e_s\right) / \beta^2 \gamma \leq
\sigma_s\left(G_A S_A\right) .
\end{equation}
$A_M$ and $G_A S_A$ are the top left $s\times s$ corners of $M$ and
$GS$, respectively. Hence, we can write $\vectornorm{A_M - G_A
S_A}_2 \leq \vectornorm{M - GS }_2 = e_s$. By combining this
expression with Eq. \eqref{eq:nonsing_left_side} and using
assumption \ref{singval_assump}, we have $\vectornorm{A_M - G_A
S_A}_2 \leq \sigma_s\left(G_A S_A\right)$. Equivalently,
\begin{equation}\label{eq:a_m_non_singular}
\frac{\vectornorm{A_M - G_A S_A}_2}{\vectornorm{G_A S_A}_2} <
\frac{1}{\kappa\left(G_A S_A\right)} .
\end{equation}
The matrix $G_A S_A$ is non-singular since it is the product of the
non-singular matrices $G_A$ and $S_A$. Equation 2.7.6 in
\cite{MatrixComputations} states that for any matrix $A$ and
perturbation matrix $\Delta A$ we have
\[\frac{1}{\kappa_2\left(A\right)} = \min_{A+\Delta A \text{ singular}} \frac{\vectornorm{\Delta A}_2}{\vectornorm{A}_2} .\]
This equation in effect gauges the minimal $L_2$ distance from $A$
to a singular matrix. By setting $G_A S_A = A$ in Eq.
\eqref{eq:a_m_non_singular} we conclude that $A_M$ is non-singular.
\end{proof}

Assumption \ref{svd_assump} can be verified for different types of
rank-$s$ approximations of $M$. For the approximated SVD we have
Corollary \ref{cor:algo_svd_bound}.
\begin{corollary}\label{cor:algo_svd_bound}
When the approximated SVD is used to form $GS$, we have $\gamma = 1$
(where $\gamma$ is defined by assumption \ref{svd_assump} in Theorem
\ref{th:a_is_invertible}).
\end{corollary}
\begin{proof}
Let $M\simeq U\Sigma V^T$ be the approximated SVD of $M$. We can
choose $G = U\Sigma$ and $S=V^T$. From the properties of the SVD, we
have $\sigma_s\left(G\right) = \sigma_s\left(U\Sigma\right) =
\Sigma_{ss} = \sigma_s\left(GS\right)$ and $\sigma_s\left(S\right) =
1$. It follows that $\sigma_s\left(G\right) \sigma_s\left(S \right)
= \sigma_s\left(GS\right)$.
\end{proof}

Similarly, the $\beta$ in assumption \ref{rrqr_assump} depends on
the algorithm that is used to pick $G_A$ and $S_A$ from within $G$
and $S$, respectively. When a state-of-the-art RRQR algorithm is
used, we derive Corollary \ref{cor:algo_rrqr_bound}.
\begin{corollary}\label{cor:algo_rrqr_bound}
When the RRQR version given in Algorithm 1 in \cite{AlgoRRQR} is
used to choose $G_A$ and $S_A$, we have $\beta \leq
\sqrt{s\left(min\left(m,n\right)-s\right) + 1}$ , where $\beta$ is
defined by assumption \ref{rrqr_assump} in Theorem
\ref{th:a_is_invertible}.
\end{corollary}
\begin{proof}
Let $A\in\mathbb{R}^{n\times k}$  be a matrix where $k\leq n$ and a
let $A = \left[\begin{array}{cc} A_1 & A_2\end{array}\right]$ be a
partition of $A$ where $A_1\in\mathbb{R}^{k\times k}$. The concept
of local $\mu$-maximum volume was used in \cite{AlgoRRQR} to find a
pivoting scheme such that $\sigma_{min}\left(A_1\right)$ is bounded
from below. Formally, Lemma 3.5 in \cite{AlgoRRQR} states that when
$A_1$ is a local $\mu$-maximum volume in $A$, we have
$\sigma_{min}\left(A_1\right) \geq \sigma_k\left(A\right) /
\sqrt{k\left(n-k\right)\mu^2 + 1}$. $\mu$ is a user-controlled
parameter that has negligible effect in this bound. For instance,
\cite{AlgoRRQR} suggests setting $\mu = 1 + \mathbf{u}$, where
$\mathbf{u}$ is the machine precision. Therefore, we omit $\mu$ in
subsequent references of this bound.

Algorithm 1 in \cite{AlgoRRQR} describes how a local $\mu$-maximum
volume can be found for a given matrix $A$. This algorithm can be
applied to the choice of $G_A$ and $S_A^T$ from the rows of $G$ and
$S^T$, respectively. It follows from Lemma 3.5 in \cite{AlgoRRQR}
that $\sigma_s\left(G_A\right) \geq \sigma_s\left(G\right) /
\sqrt{s\left(m-s\right) + 1}$ and $\sigma_s\left(S_A\right) =
\sigma_s\left(S_A^T\right) \geq \sigma_s\left(S^T\right) /
\sqrt{s\left(n-s\right) + 1} = \sigma_s\left(S\right) /
\sqrt{s\left(n-s\right) + 1}$. The definition of $\beta$ yields the
required expression.
\end{proof}
Later the RRQR algorithm will be used to select $G_A^T$ and $S_A$ as
columns from $G^T$ and $S$, respectively. This is equivalent to
choosing rows from  $G$ and $S^T$. The latter form was used for
compatibility with the notation of \cite{AlgoRRQR}.

Theorem \ref{th:a_is_invertible} states that if our rank-$s$
approximation of $M$ is sufficiently accurate and our RRQR algorithm
managed to pick $s$ non-singular columns from $G^T$ and $S$, then
our sample matrix $A_M$ is non-singular.

We bring a few definitions in order to bound the error of the
Nystr\"{o}m approximation procedure. We will decompose the matrix
$M$ into a sum of two matrices: $M_{lg}$ that contains the energy of
the top $s$ singular values and $M_{sm}$ that contains the residual.
If $M_{lg}$ and $M_{sm}$ are given in SVD outer product form, then
we have $M_{lg} = \sum_{i=1}^s \sigma_i u_i v_i$ and $M_{sm} =
\sum_{i=s+1}^{min(m,n)} \sigma_i u_i v_i$, respectively. Based on
this decomposition, we define the following decompositions of
$M_{lg}$ and $M_{sm}$:
\begin{equation}\label{eq:m_lg_sm_parts}
M = M_{lg} + M_{sm} = \left[ \begin{array}{cc}
            A_M & B_M \\
            F_M & C_M
        \end{array} \right]
= \left[ \begin{array}{cc}
            A_{lg} & B_{lg} \\
            F_{lg} & C_{lg}
        \end{array} \right]
+ \left[ \begin{array}{cc}
            A_{sm} & B_{sm} \\
            F_{sm} & C_{sm}
        \end{array} \right] .
\end{equation}
\begin{lemma}\label{lem:a_lg_invertible}
If all the assumptions of Theorem \ref{th:a_is_invertible} hold and
if we have
\begin{equation}\label{eq:m_s1_singval_bound}
\sigma_{s+1}\left(M\right) < \frac{\sigma_s\left(M\right) -
e_s}{\beta^2 \gamma} - e_s
\end{equation}
(where $e_s$ is defined by assumption \ref{singval_assump} in
Theorem \ref{th:a_is_invertible}), then $A_{lg}$ is non-singular.
\end{lemma}
\begin{proof}
We employ Lemma \ref{lem:sing_val_diff} to bound
$\abs{\sigma_s\left(A_M\right) - \sigma_s\left(G_A S_A\right)}$.
Formally, we have
\[
\abs{\sigma_s\left(A_M\right) - \sigma_s\left(G_A S_A\right)} \leq
\vectornorm{A_M - G_A S_A}_2 \leq \vectornorm{M - G_S}_2 = e_s .
\]
By rearranging terms, we obtain $\sigma_s\left(G_A S_A\right) - e_s
\leq \sigma_s\left(A_M\right)$. Combining this expression with Eq.
\eqref{eq:nonsing_left_side} from the proof of Theorem
\ref{th:a_is_invertible} yields
\begin{equation}\label{eq:am_singval_lim}
\frac{\sigma_s\left(M\right)-e_s}{\beta^2 \gamma} - e_s \leq
\sigma_s\left(A_M\right) .
\end{equation}
The quantity $\vectornorm{A_M - A_{lg}}_2$ can be bounded by
$\vectornorm{A_M - A_{lg}}_2 \leq  \vectornorm{M - M_{lg}}_2 =
\sigma_{s+1}\left(M\right)$. Combining the above with Eqs.
\eqref{eq:m_s1_singval_bound} and \eqref{eq:am_singval_lim} yields
\[
\vectornorm{A_M - A_{lg}}_2 \leq \sigma_{s+1}\left(M\right) <
\frac{\sigma_s\left(M\right)-e_s}{\beta^2 \gamma} - e_s \leq
\sigma_s\left(A_M\right) .
\]
The terms are rearranged to get
\begin{equation}\label{eq:a_lg_non_singular}
\vectornorm{A_M - A_{lg}}_2 / \vectornorm{A_M}_2 < 1 /
\kappa\left(A_M\right) ,
\end{equation}
where $\kappa$ is the standard $L_2$-norm condition number. This
expression is similar to Eq. \eqref{eq:a_m_non_singular} in the
proof of Theorem \ref{th:a_is_invertible}. As before, if $A_M$ is
non-singular, then Eq. \eqref{eq:a_lg_non_singular} implies that
$A_{lg}$ is non-singular.
\end{proof}

We define the rank-$s$ approximation of $M$ that is based on the
truncated SVD form of $M_{lg}$. Let $M_{lg} = U_s \Sigma_s V_s^T$ be
the truncated SVD of $M$. Denote $X=U_s \Sigma_s$ and $Y=V_s^T$ such
that $M_{lg} = XY$. We define $X = \left[ \begin{array}{cc} X_A &
X_B \end{array} \right]^T$ and $Y = \left[ \begin{array}{cc} Y_A &
Y_B \end{array} \right]$ where $X_A, Y_A \in {\mathbb R}^{s\times
s}$. We get the following forms for the components of $M_{lg}$:
$A_{lg}= X_A Y_A,\ \ B_{lg}= X_AY_B,\ \ F_{lg}= X_BY_A$ and $C_{lg}
= X_BY_B$.

The Nystr\"{o}m approximation error can now be formulated.
\begin{lemma}\label{lem:span_approx_error}
Assume that $A_M$ and $A_{lg}$ are non-singular. Then, the error of
the Nystr\"{o}m approximation procedure is bounded by
\begin{equation}\label{eq:span_approx_error}
\frac{\sigma_{s+1}\left(M\right)}{\sigma_{s}\left(A_M\right)} \left(
\frac{\sigma_1\left(M\right)^2}{\sigma_s\left(A_{lg}\right)} +
2\sigma_1\left(M\right) + \sigma_{s+1}\left(M\right) \right) .
\end{equation}
\end{lemma}
\begin{proof}
As seen from Eq. \eqref{eq:m_hat_general_form}, the matrices $A_M,
B_M$ and $F_M$ are not modified by the Nystr\"{o}m extension. $C_M$
is approximated as $F_M A_M^{+} B$. Assuming that $A$ is
non-singular, then $F_M A_M^{+} B$ is equivalent to $F_M A_M^{-1}
B$. The latter can be decomposed using the partitioning in Eq.
\eqref{eq:m_lg_sm_parts}:
\begin{equation}\label{eq:error_major_div}
\begin{array}{lc}
F_M A_M^{-1} B = \left(F_{lg} + F_{sm}\right)A_M^{-1} \left(B_{lg} + B_{sm}\right) = \\
 = F_{lg} A^{-1} B_{lg} + F_{lg} A^{-1} B_{sm} + F_{sm} A^{-1} B_{lg} + F_{sm} A^{-1} B_{sm} .
\end{array}
\end{equation}
Since $A_M$ and $A_{lg}$ are non-singular, we have $A_M^{-1} -
A_{lg}^{-1} = -A_{lg}^{-1}\left(A - A_{lg}\right)A_M^{-1} =
-A_{lg}^{-1} A_{sm} A_M^{-1}$. The first term of Eq.
\eqref{eq:error_major_div} can be written as
\begin{equation}\label{eq:error_minor_div}
F_{lg} A^{-1} B_{lg} = F_{lg} \left(A_{lg}^{-1} - A_{lg}^{-1} A_{sm}
A_M^{-1}\right) B_{lg} = F_{lg} A_{lg}^{-1} B_{lg} - F_{lg}
A_{lg}^{-1} A_{sm} A_M^{-1} B_{lg} .
\end{equation}
By our assumption, the matrices $X_A$ and $Y_A$ are non-singular
since $A_{lg} = X_A Y_A$ is non-singular. The first term of Eq.
\eqref{eq:error_minor_div} becomes:
\[
F_{lg} A_{lg}^{-1} B_{lg} = X_B Y_A \left(X_A Y_A\right)^{-1} X_A
Y_B = X_B Y_A Y_A^{-1} X_A^{-1} X_A Y_B = X_B Y_B = C_{lg} .
\]
This means that $F_{lg} A_{lg}^{-1} B_{lg}$ is the best rank-$s$
approximation to $C_M$, as given by the truncated SVD of $M$. We can
bound the error by collecting all the other terms in Eqs.
\eqref{eq:error_major_div} and \eqref{eq:error_minor_div}:
\[
E_{nys} = - F_{lg} A_{lg}^{-1} A_{sm} A_M^{-1} B_{lg} + F_{lg}
A^{-1} B_{sm} + F_{sm} A^{-1} B_{lg} + F_{sm} A^{-1} B_{sm} .
\]
By the definition of $M_{sm}$ in Eq. \eqref{eq:m_lg_sm_parts}, we
have $\vectornorm{M_{sm}}_2 \leq \sigma_{s+1}\left(M\right)$.
Therefore, we can bound
$\vectornorm{A_{sm}}_2,\vectornorm{B_{sm}}_2$ and
$\vectornorm{F_{sm}}_2$ by $\sigma_{s+1}\left(M\right)$. Similarly,
$\vectornorm{B_{lg}}_2$ and $\vectornorm{F_{lg}}_2$ are bounded by
$\sigma_1\left(M\right)$. The overall bound on
$\vectornorm{E_{nys}}_2$ is
\[
\begin{array}{c}
\vectornorm{E_{nys}}_2 = \vectornorm{- F_{lg} A_{lg}^{-1} A_{sm} A_M^{-1} B_{lg} + F_{lg} A^{-1} B_{sm} + F_{sm} A^{-1} B_{lg} + F_{sm} A^{-1} B_{sm}}_2 \leq \\ \vectornorm{F_{lg} A_{lg}^{-1} A_{sm} A_M^{-1} B_{lg} }_2 + \vectornorm{F_{lg} A^{-1} B_{sm} }_2 + \vectornorm{ F_{sm} A^{-1} B_{lg} }_2 + \vectornorm{F_{sm} A^{-1} B_{sm}}_2 \leq \\
\frac{\sigma_1\left(M\right)^2
\sigma_{s+1}\left(M\right)}{\sigma_s\left(A_M\right)\sigma_s\left(A_{lg}\right)}
+
\frac{\sigma_1\left(M\right)\sigma_{s+1}\left(M\right)}{\sigma_s\left(A_M\right)}
+
\frac{\sigma_1\left(M\right)\sigma_{s+1}\left(M\right)}{\sigma_s\left(A_M\right)}
+
\frac{\sigma_{s+1}\left(M\right)^2}{\sigma_s\left(A_M\right)} = \\
\frac{\sigma_{s+1}\left(M\right)}{\sigma_{s}\left(A_M\right)} \left(
\frac{\sigma_1\left(M\right)^2}{\sigma_s\left(A_{lg}\right)} +
2\sigma_1\left(M\right) + \sigma_{s+1}\left(M\right) \right) .
\end{array}
\]
\end{proof}

Corollary \ref{cor:rank_s_matrix_approx} is derived
straightforwardly:
\begin{corollary}\label{cor:rank_s_matrix_approx}
If $A_M$ is non-singular and the matrix $M$ is rank-$s$, then, the
Nystr\"{o}m extension approximates $M$ perfectly.
\end{corollary}
\begin{proof}
If $M$ is rank-$s$ then $A_{lg} = A_M$ and the conditions in Lemma
\ref{lem:span_approx_error} hold. We obtain the result by setting
$\sigma_{s+1}\left(M\right) = 0$ in Eq.
\eqref{eq:span_approx_error}.
\end{proof}

We proceed to express the Nystr\"{o}m approximation error in
relation to the parameters $\beta, \gamma$ and $e_s$, as defined by
the assumptions in Theorem \ref{th:a_is_invertible}.
\begin{theorem}\label{th:algo_approx_error}
Assume that the assumptions of Theorem \ref{th:a_is_invertible} hold
as well as the assumptions of Lemmas \ref{lem:a_lg_invertible} and
\ref{lem:span_approx_error}. The error term of the Nystr\"{o}m
procedure is bounded by:
\begin{equation}\label{eq:algo_span_approx_error}
\frac{\sigma_{s+1}\left(M\right)\beta^2\gamma
}{\sigma_{s}\left(M\right) - \left(1 + \beta^2\gamma\right)e_s}
\left( \frac{\sigma_1\left(M\right)^2
\beta^2\gamma}{\sigma_s\left(M\right) - \left(1 +
\beta^2\gamma\right)e_s - \sigma_{s+1}\left(M\right)\beta^2\gamma} +
2\sigma_1\left(M\right) + \sigma_{s+1}\left(M\right) \right) .
\end{equation}
\end{theorem}
\begin{proof}
We use Lemma \ref{lem:sing_val_diff} to obtain:
\[
\abs{\sigma_s\left(A_M\right) - \sigma_s\left(A_{lg}\right)} \leq
\vectornorm{A_M - A_{lg}}_2 \leq \vectornorm{M - M_{lg}}_2 =
\sigma_{s+1}\left(M\right) .
\]
Equivalently, $\sigma_s\left(A_M\right) - \sigma_{s+1}\left(M\right)
\leq \sigma_s\left(A_{lg}\right)$. We substitute
$\sigma_s\left(A_M\right)$ with the left side of Eq.
\eqref{eq:am_singval_lim} to get
\begin{equation}\label{eq:lim_a_gl}
\frac{\sigma_s\left(M\right)-e_s}{\beta^2 \gamma} - e_s -
\sigma_{s+1}\left(M\right)\leq \sigma_s\left(A_{lg}\right) .
\end{equation}
The result follows when the expressions for
$\sigma_s\left(A_M\right)$ and $\sigma_s\left(A_{lg}\right)$ in Eq.
\eqref{eq:span_approx_error} are replaced with the left sides of
Eqs. \eqref{eq:am_singval_lim} and \eqref{eq:lim_a_gl},
respectively.
\end{proof}

When $A_M$ is non-singular, the eigengap in the $s^{th}$ singular
value governs the approximation error. This can be seen from Eq.
\eqref{eq:algo_span_approx_error}, where the eigengap appears in the
expression $\frac{\sigma_{s+1}\left(M\right)\beta^2\gamma
}{\sigma_{s}\left(M\right) - \left(1 + \beta^2\gamma\right)e_s}$.
Theorem \ref{th:algo_approx_error} bounds the general case.
Corollary \ref{cor:rank_s_matrix_approx} shows what happens in the
limit case when the eigengap is infinite.

\section{ Sample Selection Algorithm }
\label{sec:algorithm} Our algorithm is based on Theorem
\ref{th:a_is_invertible} and Corollaries \ref{cor:algo_svd_bound}
and \ref{cor:algo_rrqr_bound}. It receives as its input a matrix
$M\in {\mathbb R}^{m\times n}$ and a parameter $s$ that determines
the sample size. It returns $A_M$ - a ``good'' sub-sample of $M$. If
the algorithm succeeds, we can use Theorem
\ref{th:algo_approx_error} to bound the approximation error. The
algorithm is described in Algorithm 1.

\begin{algorithm}[H]
\caption{$\left(M,s\right)$}

\begin{enumerate}
\item  \label{approx_mat_decomp_step}Form a rank-$s$ decomposition of $M$. Formally $M\simeq GS$, where $G\in {\mathbb R}^{m\times s}$ and $S\in {\mathbb R}^{s\times n}$.

\item  \label{first_rrqr_step}Apply the RRQR algorithm to $G^T$ to find a column pivoting matrix $E_G$ such that $\left[ \begin{array}{cc} G_A^T & G_B^T\end{array} \right] = G^T E_G = Q_G R_G$, where $G_A\in {\mathbb R}^{s\times s}$ and $G_B\in {\mathbb R}^{s\times m-s}$. Let $I_s$ be the group of indices in $M$ that correspond to the first $s$ columns of $E_G$.

\item  \label{second_rrqr_step}Apply the RRQR algorithm to $S$ to find a column pivoting matrix $E_S$ such that $\left[ \begin{array}{cc} S_A & S_B\end{array} \right] = S E_S = Q_S R_S$, where $S_A\in {\mathbb R}^{s\times s}$ and $S_B\in {\mathbb R}^{s\times n-s}$. Let $J_s$ be the group of indices in $M$ that correspond to the first $s$ columns of $E_S$.

\item \begin{algorithmic}
        \IF{$rank\left( G_A \right) \neq s$ or $rank\left( S_A \right) \neq s$ }
        \STATE return ``Algorithm failed. Please pick a different value for $s$.''
        \ENDIF
      \end{algorithmic}

\item  Form the matrix $A_M\in {\mathbb R}^{s\times s}$ such that $A_M = \left[ M_{ij} \right]_{i\in I_s, j\in J_s}$. Returns $A_M$ as the sub-sample matrix.
\end{enumerate}
\end{algorithm}

\subsection{ Algorithm Complexity Analysis }

Step \ref{approx_mat_decomp_step} is the computational bottleneck of
the algorithm and can take up to
$O\left(min\left(mn^2,nm^2\right)\right)$ operations if full SVD is
used. Approximate SVD algorithms are typically faster. For example,
the algorithm in \cite{SVDHarPeled} runs in $O\left(mn\right)$ time,
which is linear in the number of elements in the matrix. If we have
some prior knowledge about the structure of the matrix, it can take
even less time. For example, if an approximation of the norms of the
columns is known, we can use $LinearTimeSvd$ \cite{LinearTimeSvd} to
achieve a sub-linear runtime complexity of $O\left(s^2 m +
s^3\right)$. We denote the runtime complexity of this step by
$T_{approx}$. Using the $RRQR$ algorithm in \cite{StrongRRQR}, steps
\ref{first_rrqr_step} and \ref{second_rrqr_step} in Algorithm 1 take
$O(ms^2)$ and $O(ns^2)$ operations, respectively. Finally, the
formation of $A_M$ takes $O(s^2)$ time. The total runtime complexity
becomes $O\left(T_{approx} + \left(m+n\right)s^2\right)$ and it is
usually dominated by $O\left(T_{approx}\right)$.

Denote the space requirements of step \ref{approx_mat_decomp_step}
in Algorithm 1 by $S_{approx}$. Then, the total space complexity
becomes $O\left(S_{approx} +s\left(m+n\right)\right)$. Typically, a
total of $O\left(\left(m+n\right)s^{O(1)}\right)$ space is used.

\subsection{ Relation to ICD }

Let $M$ be decomposed into $M=X^T X$ where $X\in{\mathbb R}^{n\times
n}$. In this case, the $R$ factor in the QR decomposition of $X$ is
the Cholesky factor of $M$ since $X=QR$ means that $M=X^T X= R^T Q^T
Q R = R^T R$. Similarly, the Cholesky decomposition of a
symmetrically pivoted $M$ corresponds to a column pivoted QR of $X$.
The pivoting strategy used by the Cholesky algorithm in the ICD
algorithm is the greedy scheme of the classical pivoted-QR algorithm
in \cite{BusingerQr}. Applying ICD to $M$ gives the $R$ factor of
the pivoted QR on $X$, and vice versa. The special structure of the
matrix enables the ICD to unite steps
\ref{approx_mat_decomp_step},\ref{first_rrqr_step} and
\ref{second_rrqr_step} in Algorithm 1, creating a rank-$s$
approximation to $M$ while at the same time choosing pivots
according to a greedy QR criterion. This allows the ICD to achieve a
runtime complexity of $O\left(s^2 n\right)$.

\section{ Experimental Results }
\label{sec:experiments} In our experiments, we employ a fast but
inaccurate sub-linear SVD approximation for step
\ref{approx_mat_decomp_step} in Algorithm 1. This approximated SVD
first randomly samples the columns of the matrix. Then, it uses
these columns in the $LinearTimeSVD$ algorithm of
\cite{LinearTimeSvd} to compute an SVD approximation in $O\left(s^2
m + s^3\right)$ operations. For this SVD algorithm, the total
runtime complexity of Algorithm 1 is $O\left(s^2 \left(m+n\right) +
s^3\right)$ which is dominated by $O\left(s^2\left(m+n\right)\right)$.

\subsection{ Kernel Matrices }
\label{sec:kernel_mat_exps} First, we compare between the
performance of Algorithm 1 and the state-of-the-art sample selection
algorithms for kernel matrices. We construct a kernel matrix for a
given dataset, then each algorithm is used to choose a fixed sized
sample. From the notation of Eqs. \eqref{eq:assym_m_decomposition}
and \eqref{eq:assym_matrix_nys_form}, the error is displayed as
$\vectornorm{\hat{M}-M}$.

The following algorithms were compared: 1. The ICD algorithm
presented in section \ref{sec:gram_matrices_related_work}; 2. The
\emph{k}-means based algorithm presented in section
\ref{sec:gram_matrices_related_work}; 3. Random choice of sub-sample
as given in \cite{Chung}; 4. $LinearTimeSVD$ of
\cite{LinearTimeSvd}; 5. Algorithm 1; 6. SVD. The SVD algorithm is
used as a benchmark, since it provides rank-$s$ approximation with
the lowest Frobenius norm error. The empirical gain of our procedure
can be measured by the difference between the approximation errors
of $LinearTimeSVD$ and Algorithm 1, since $LinearTimeSVD$ is used in
Step  \ref{approx_mat_decomp_step} of Algorithm 1.

We use a Gaussian kernel of the form $k(x,y) =
exp\left(-\vectornorm{x-y}^2 / \epsilon\right)$ where $\epsilon$ is
the average squared distance between data points and the means of
each dataset. Results for methods which contain probabilistic
components are presented as the averages over 20 trials. These
include methods 2, 3, 4 and 5. The sample size is gradually
increased from 1\% to 10\% of the total data and the error is
measured in terms of the Frobenius norm.  The benchmark datasets,
summarized in Table \ref{summary}, were taken from the LIBSVM
archive \cite{DataSets}. The overall experimental parameters were
chosen to allow for comparison with Fig. 1 in \cite{SampleKmeans}.

The results are presented in Fig. \ref{fig:kernel_approx_errors}.
Algorithm 1 generally outperforms the  random sample selection
algorithm, particularly on datasets with fast spectrum decay such as
\textit{german.numer, segment} and \textit{svmguide1a}. In these datasets,
our algorithm approaches and sometimes even surpasses the
state-of-the-art \emph{k}-means based algorithm of
\cite{SampleKmeans}. This fits our derivation  for the approximation
error given by Theorem \ref{th:algo_approx_error}.

It should be noted that the algorithm in [26] has a runtime complexity
of $O\left(sn\right)$ compared to our $O\left(s^2n\right)$ for this setting.
This difference has no real-world consequences when $s$ is very small or
even constant, as typical for these problems.

In some cases, Algorithm 1 actually performs worse than
$LinearTimeSVD$. We use a greedy RRQR algorithm which sometimes
does not properly sort the singular-vectors according to their importance
(namely, the absolute value of the singular-value). This can happen
for instance when the spectrum decays slowly, which means leading singular
values are close in magnitude. In Algorithm 1, we always choose the
top $s$ indices as found by the RRQR algorithm, so we might get things wrong.

\begin{table}[H]
    \begin{center}
      \begin{tabular}{l cccccccc}
        \hline
        \small{dataset}      & \textbf{german.numer} & \textbf{splice} & \textbf{adult1a} & \textbf{dna}  & \textbf{segment} & \textbf{w1a}  & \textbf{svmgd1a} & \textbf{satimage} \\
        \small{sample count} & 1000   & 1000   & 1605    & 2000 & 2310    & 2477 & 3089    & 4435 \\
        \small{dimension}    & 24     & 60     & 123     & 180  & 19      & 300  & 4       & 36   \\
        \hline
      \end{tabular}
    \end{center}
    \vspace{-15pt}
    \caption{Summary of benchmark datasets (taken from \cite{DataSets}) }
    \label{summary}
\end{table}

\begin{figure}[h]
    \begin{center}
\begin{tabular}{ccc}
%\includegraphics[width=0.31\columnwidth]{./fig1/fig1_1.pdf} & \includegraphics[width=0.31\columnwidth]{./fig1/fig1_2.pdf} &
%\includegraphics[width=0.31\columnwidth]{./fig1/fig1_3.pdf}\tabularnewline \includegraphics[width=0.31\columnwidth]{./fig1/fig1_4.pdf}&
%\includegraphics[width=0.31\columnwidth]{./fig1/fig1_5.pdf} & \includegraphics[width=0.31\columnwidth]{./fig1/fig1_6.pdf}\tabularnewline
%\end{tabular}
%\begin{tabular}{cc}
%\includegraphics[width=0.31\columnwidth]{./fig1/fig1_7.pdf} & \includegraphics[width=0.31\columnwidth]{./fig1/fig1_8.pdf}\tabularnewline
\includegraphics[width=0.31\columnwidth]{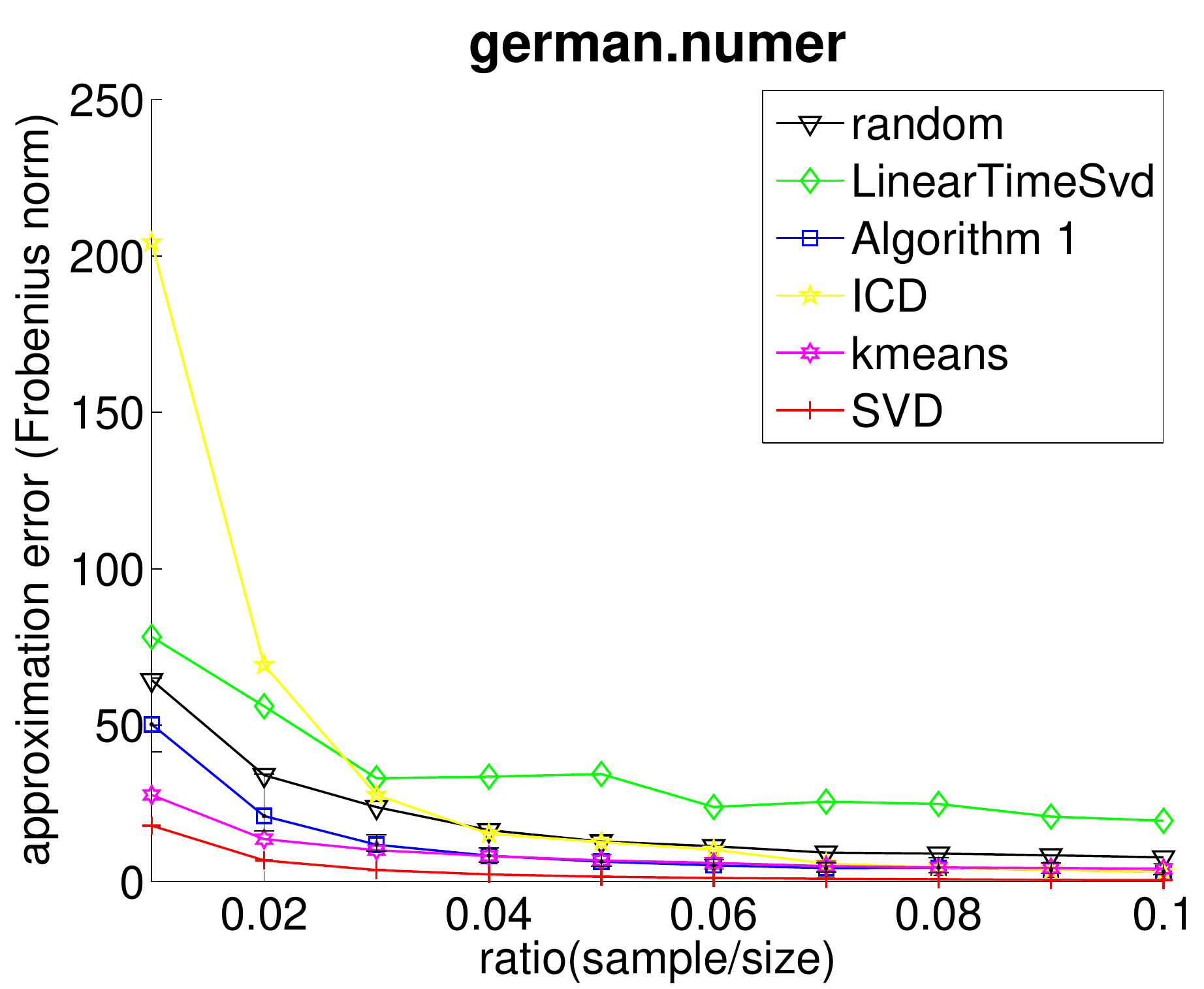} & 
\includegraphics[width=0.31\columnwidth]{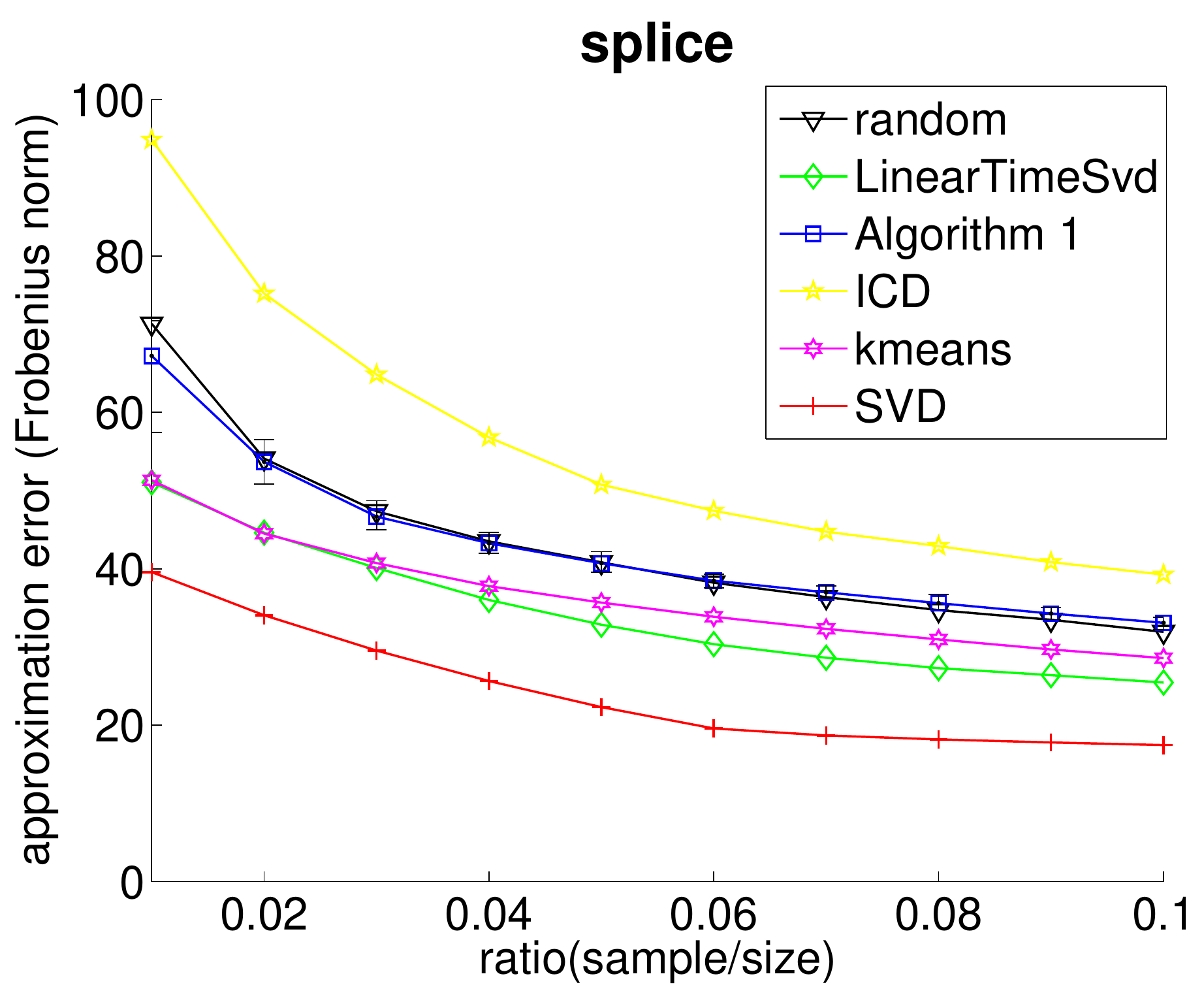} &
\includegraphics[width=0.31\columnwidth]{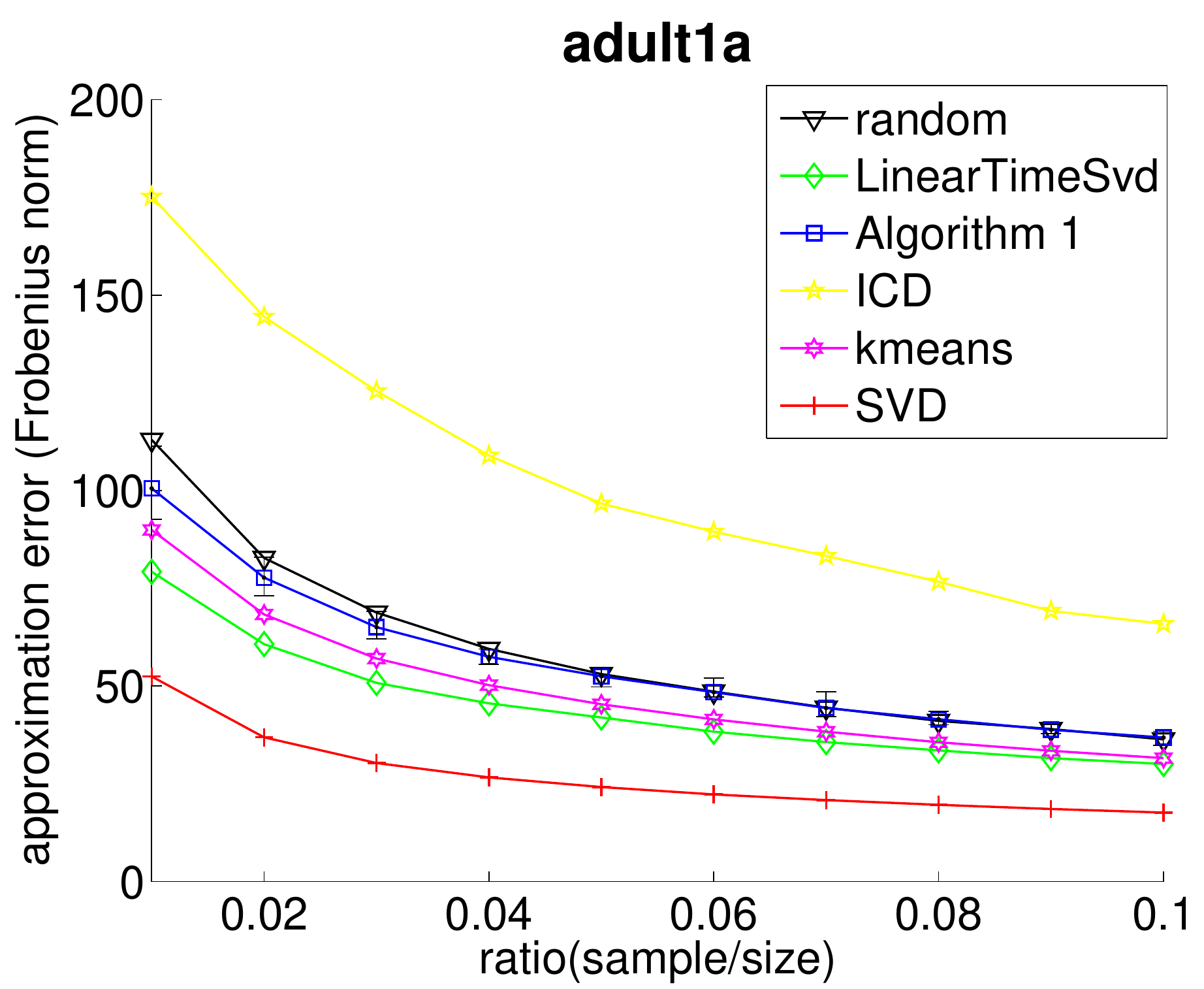}\tabularnewline \includegraphics[width=0.31\columnwidth]{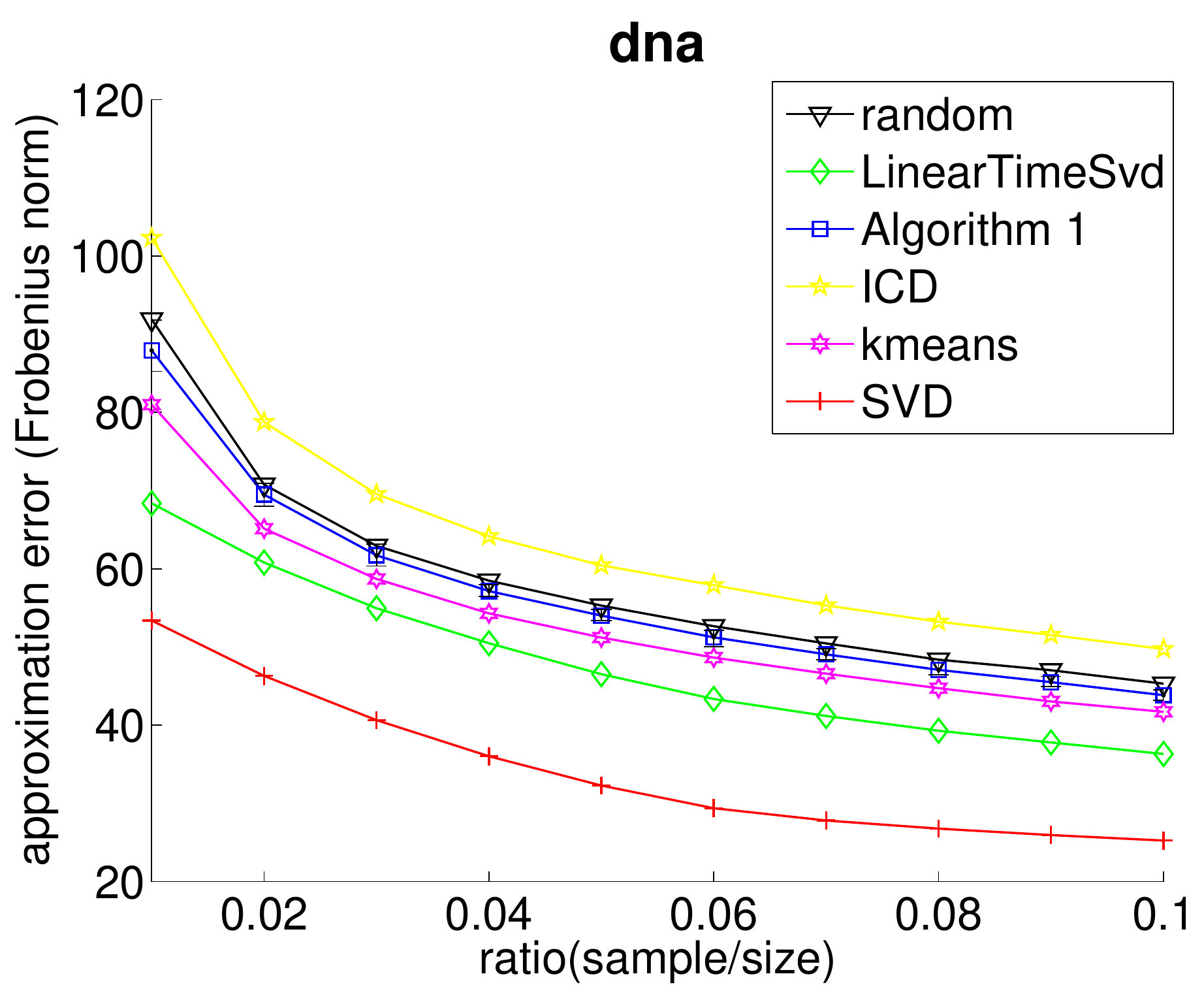}&
\includegraphics[width=0.31\columnwidth]{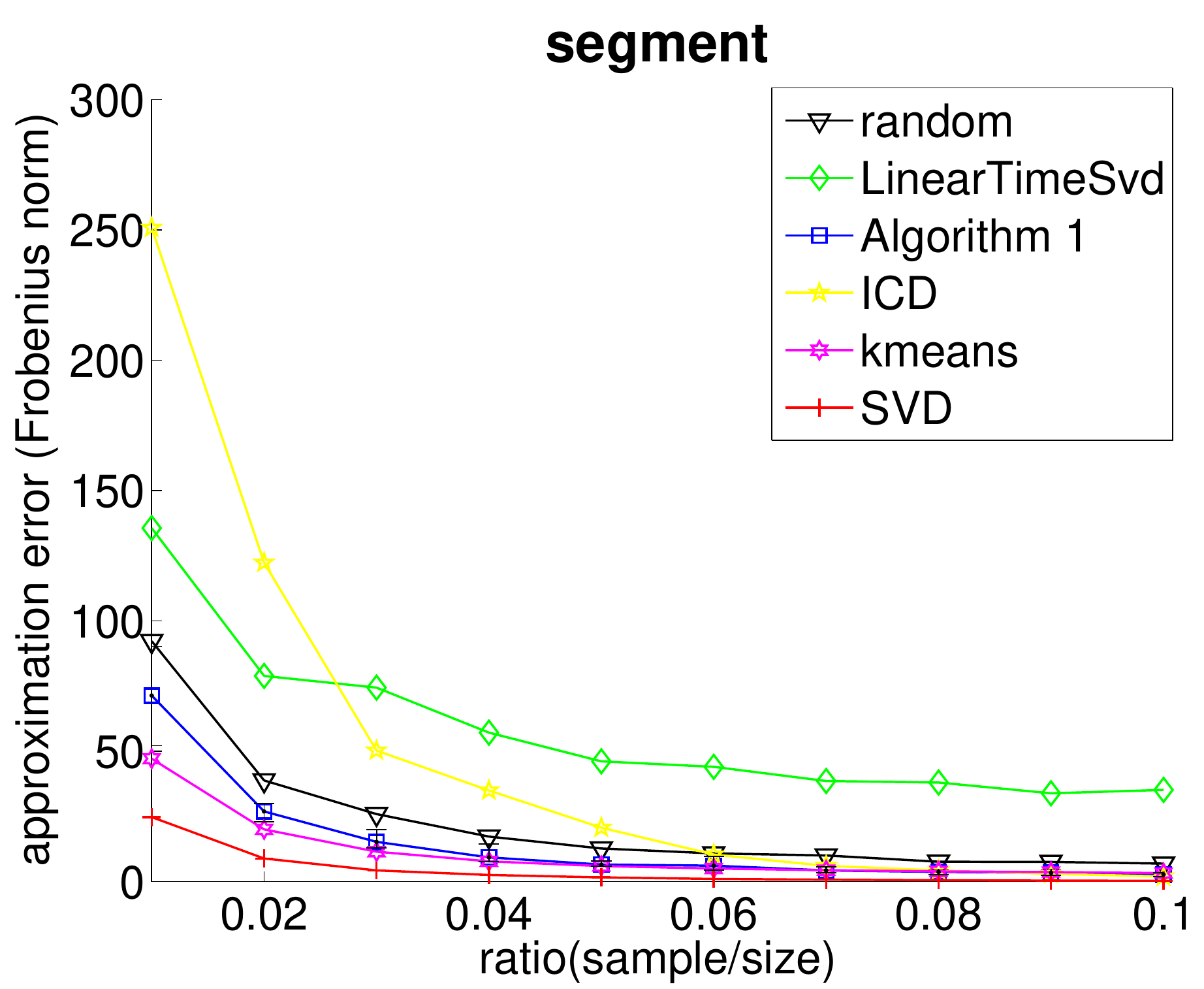} & 
\includegraphics[width=0.31\columnwidth]{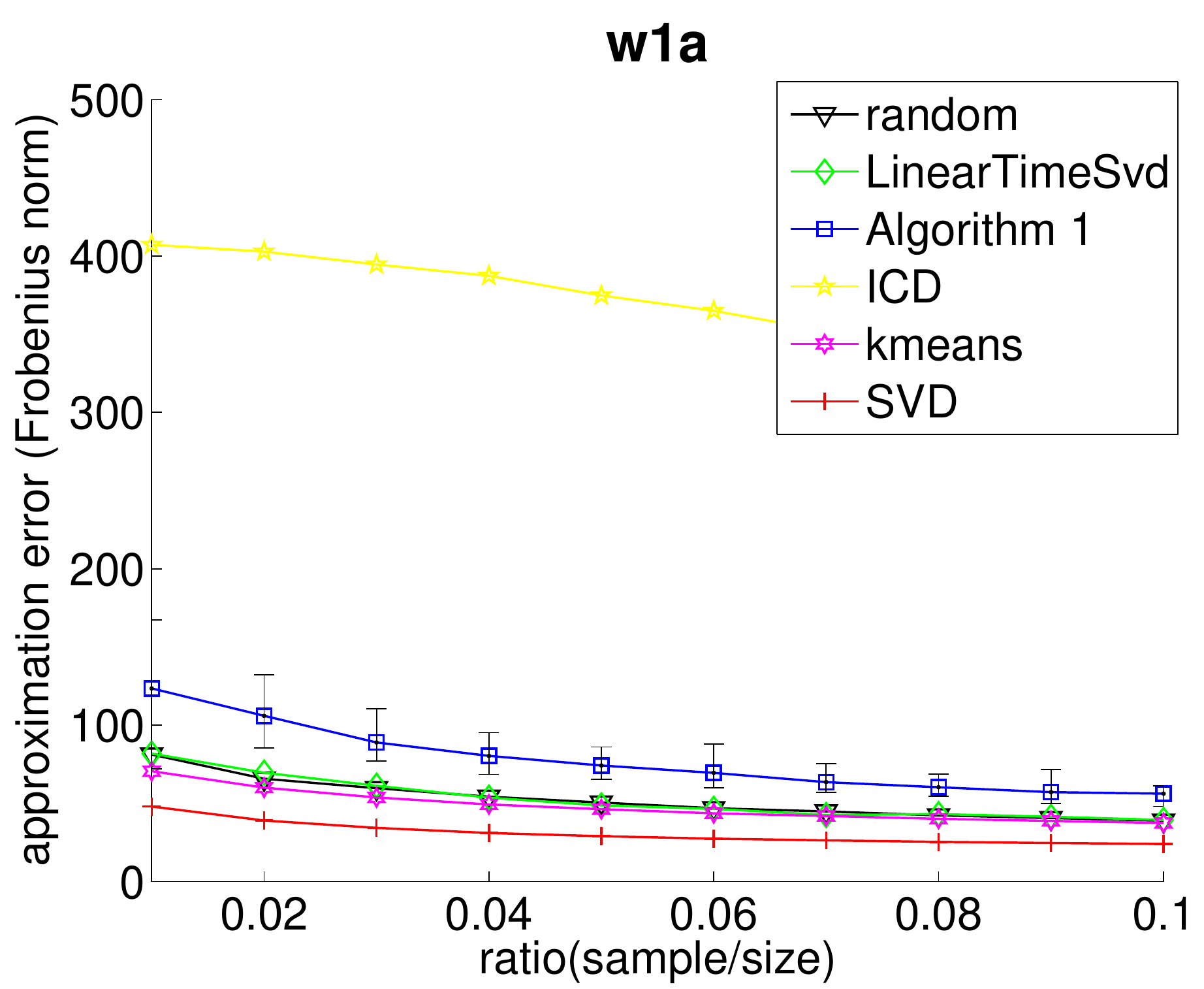}\tabularnewline
\end{tabular}
\begin{tabular}{cc}
\includegraphics[width=0.31\columnwidth]{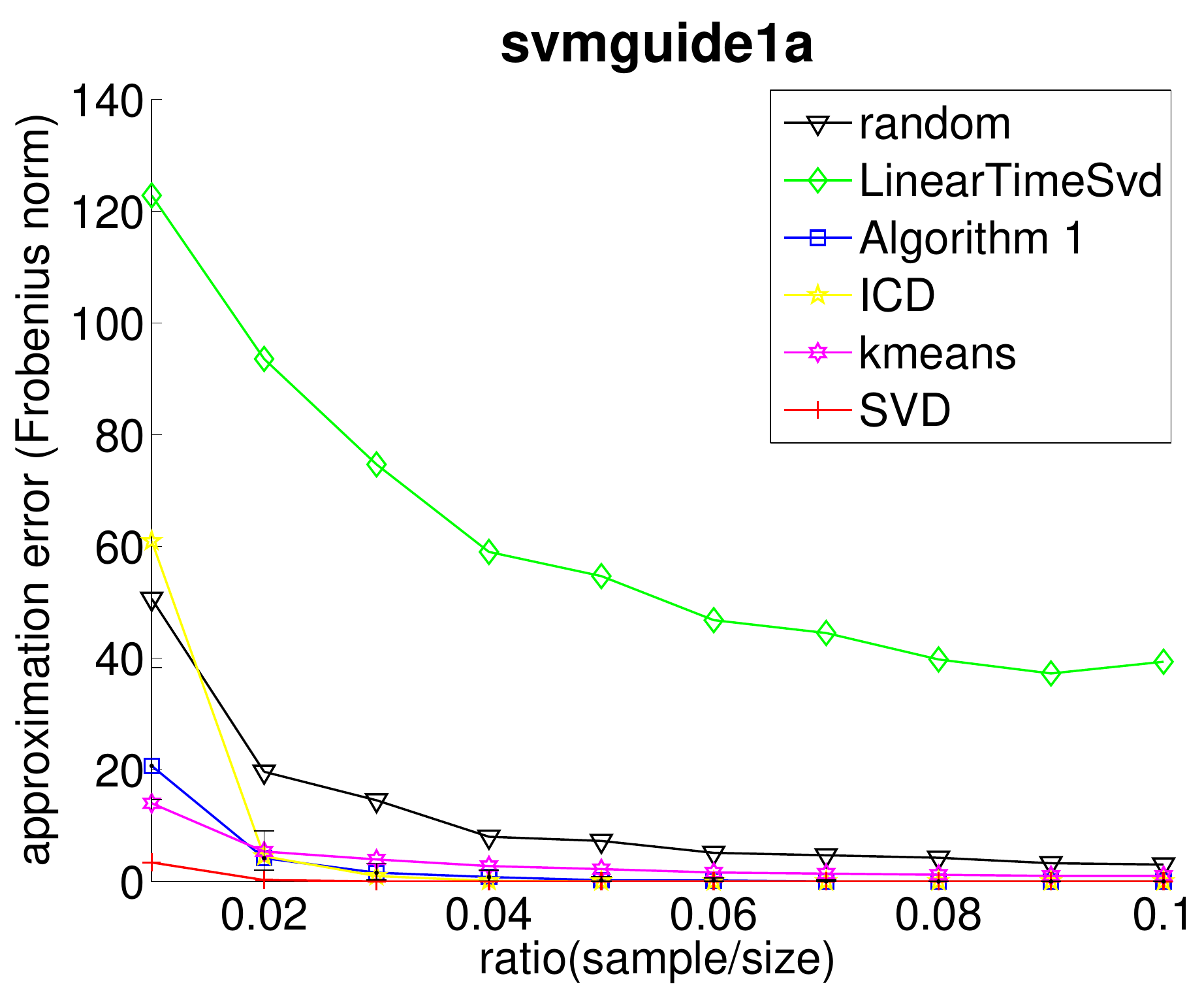} & 
\includegraphics[width=0.31\columnwidth]{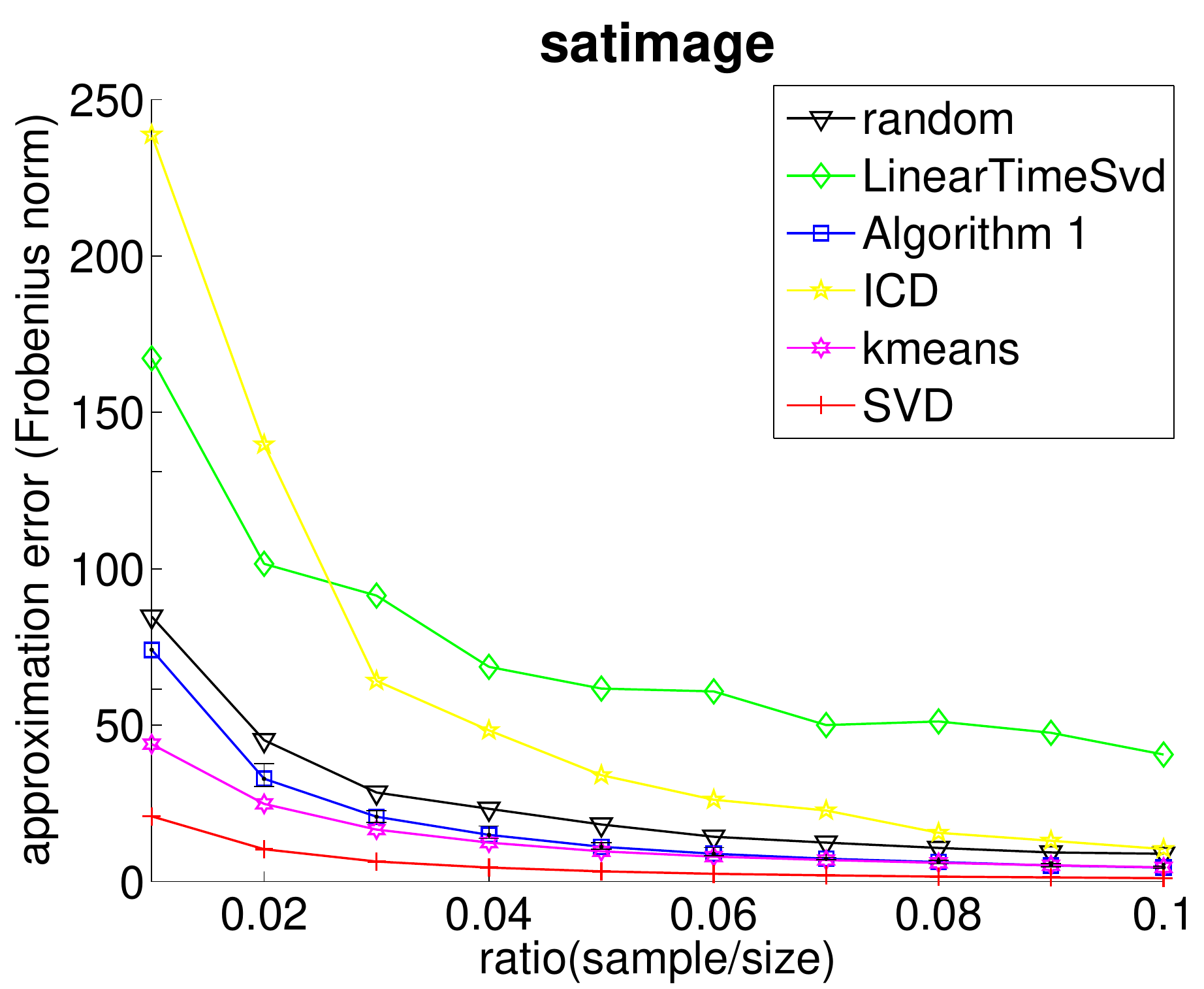}\tabularnewline
\end{tabular}
    \end{center}
    \vspace{-20pt}
    \caption{Nystr\"{o}m approximation errors for kernel matrices. The X-axis is the sampling ratio given as sample size divided by the matrix size.
    The Y-axis is the approximation error given in Frobenius norm. The tested algorithms are: random,
    LinearTimeSVD, Algorithm 1, ICD, k-means and SVD}
    \label{fig:kernel_approx_errors}
\end{figure}

%\begin{figure}[H]
%\includegraphics[width=9cm,height=7cm]{New-figures/fig1/fig1_11.pdf}
%\includegraphics[width=9cm,height=7cm]{New-figures/fig1/fig1_22.pdf}
%\includegraphics[scale=0.4]{New-figures/fig1/fig1_11.pdf}
%\includegraphics[scale=0.4]{New-figures/fig1/fig1_22.pdf}
% \caption{Nystr\"{o}m approximation errors for kernel matrices. The X-axis is the sampling ratio given as sample size divided by the matrix size.
%    The Y-axis is the approximation error given in Frobenius norm. }
%    \label{fig:kernel_approx_errors}
%\end{figure}

\subsection{ General Matrices }
\label{sec:general_mat_exp} We evaluate the performance of Algorithm
1 on general matrices by comparing it to a random choice of
sub-sample. We use the full SVD as a benchmark that theoretically
achieves the
 best accuracy. The approximation error is measured by
$\vectornorm{\hat{M}-M}_2$.

The testing matrices in this section were chosen to have non-random
spectra with random singular subspaces. Initially, a non-random
diagonal matrix $L$ is chosen with non-increasing diagonal entries.
$L$ will serve as the spectrum of our testing matrix. Then, two
random unitary matrices $U$ and $V$ are generated. Our testing
matrix is formed by $ULV^T$. We examine two degrees of spectrum
decay: linear decay (slow) and exponential decay (fast).

The error is presented in $L_2$ norm and we vary the sample size to
be between 1\%-10\% of the matrix size. The presented results are
from an averaging of 20 iterations to reduce the statistical
variability. For simplicity, we produce results only for $500\times
500$ square matrices.

The results are presented in Fig. \ref{fig:general_approx_errors}.
When the spectrum decays slowly, Algorithm 1 has no advantage over
random sample selection. It produces overall pretty bad results. But
the situation is much different in the presence of a fast spectrum
decay. Algorithm 1 displays good results when the sample size allows
it to capture most of the significant singular values of the data
(at a sample rate of about 3\%). It is interesting to note that
random sample selection does not lag far behind. This hints that, on
average, any sample is a good sample as long as it captures more
data than the numeric rank of the matrix.

\begin{figure}[h]
    \begin{center}
\begin{tabular}{cc}
\includegraphics[width=0.47\columnwidth]{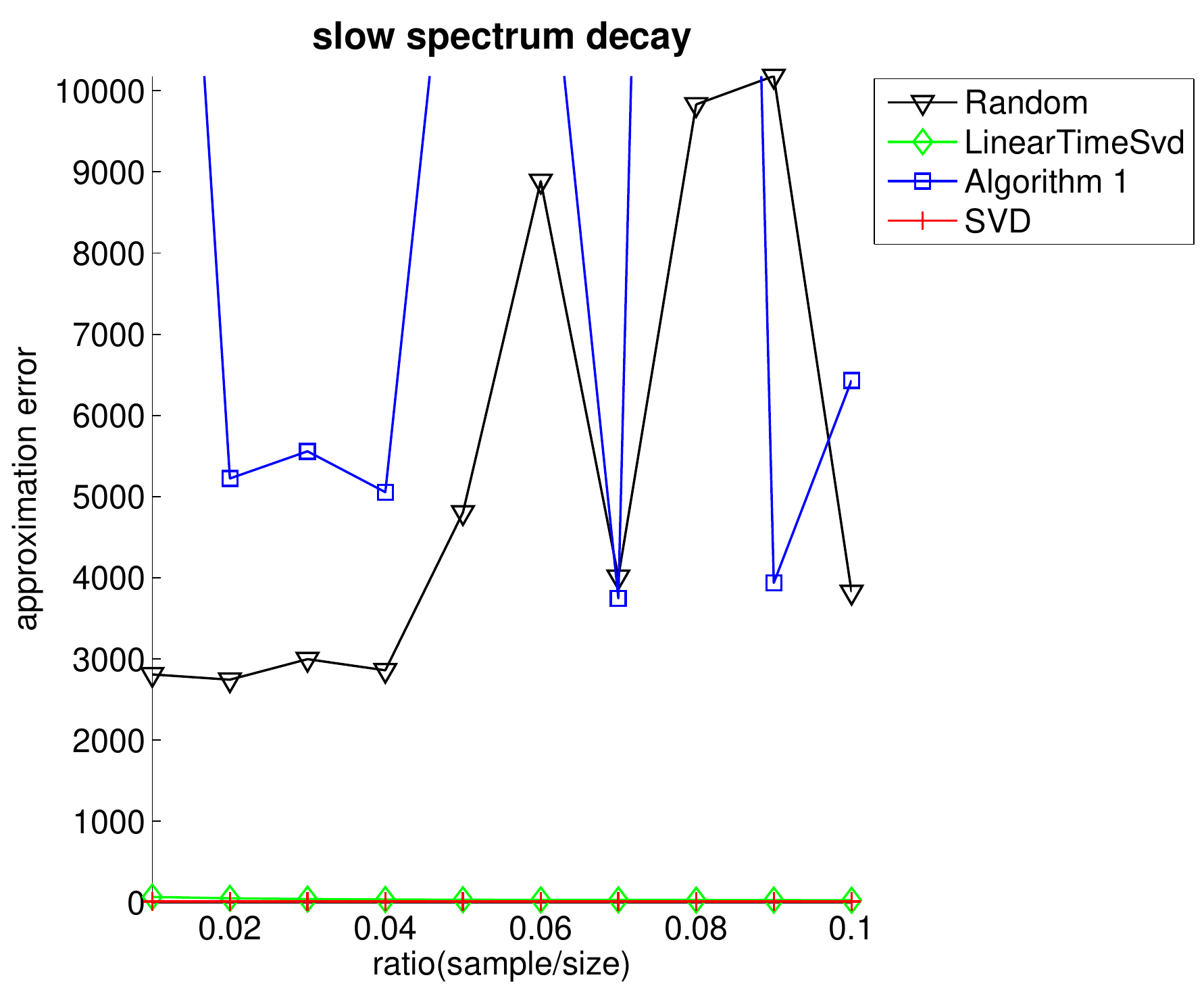} & 
\includegraphics[width=0.47\columnwidth]{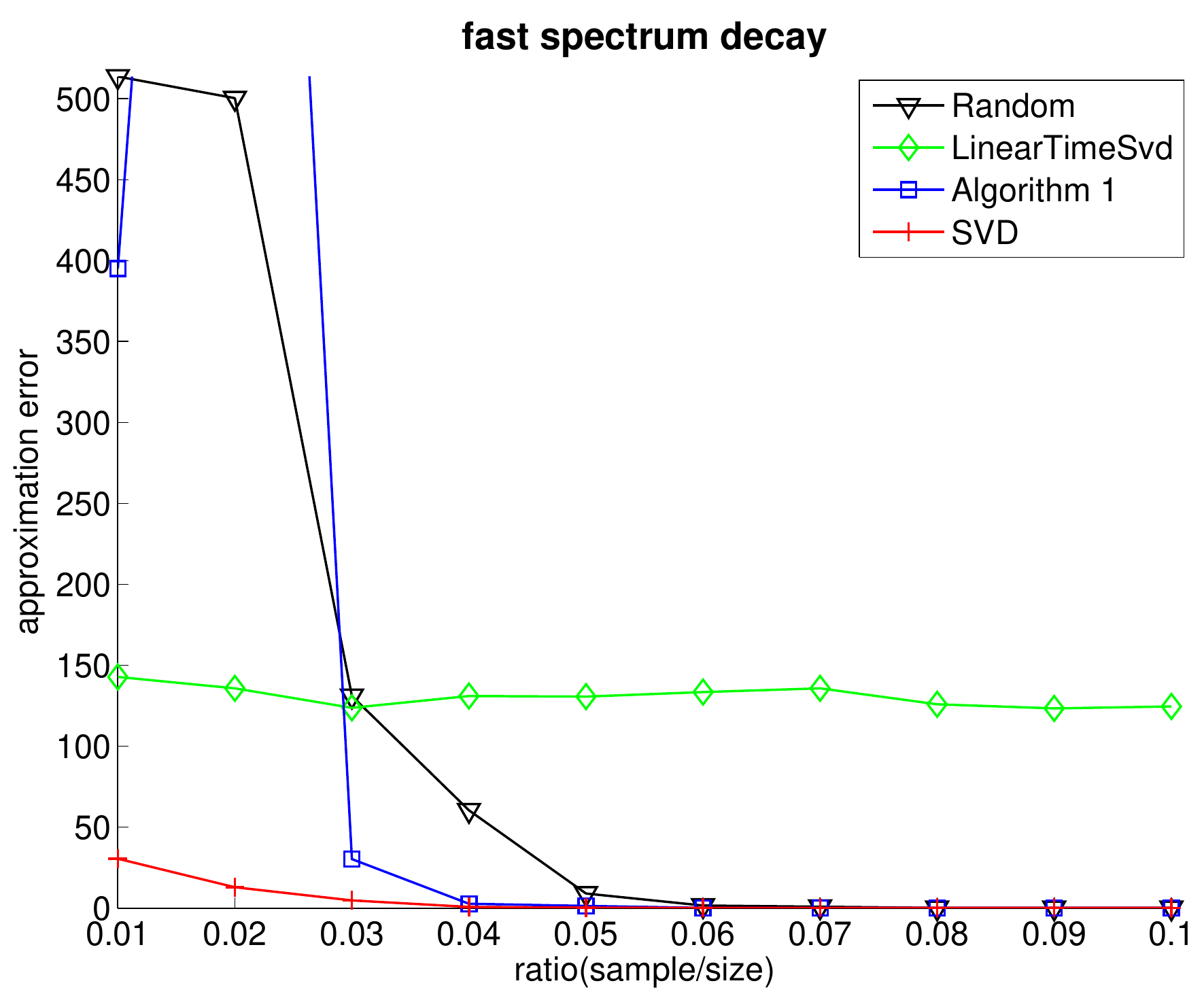}\tabularnewline
\includegraphics[width=0.47\columnwidth]{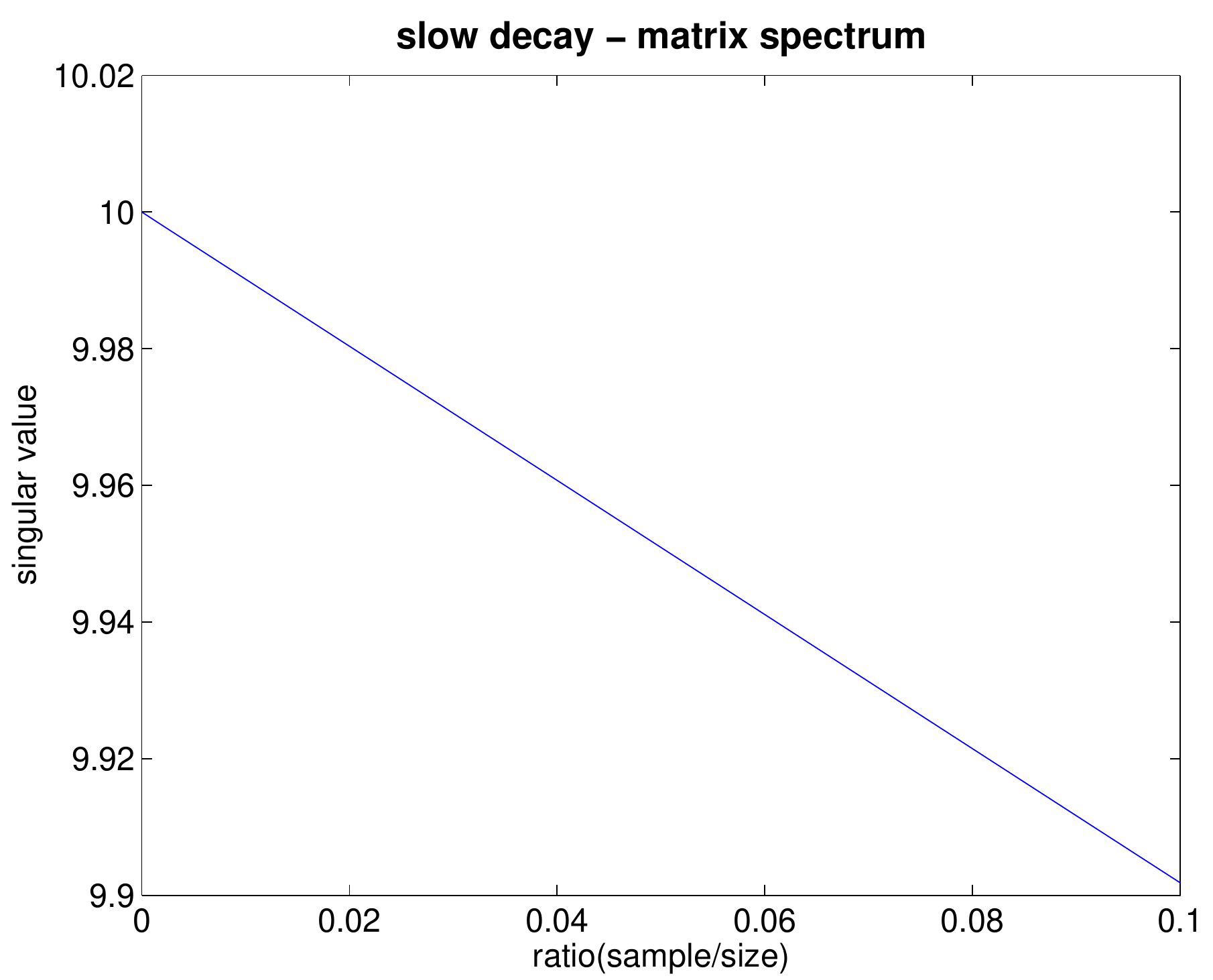} & \includegraphics[width=0.462\columnwidth]{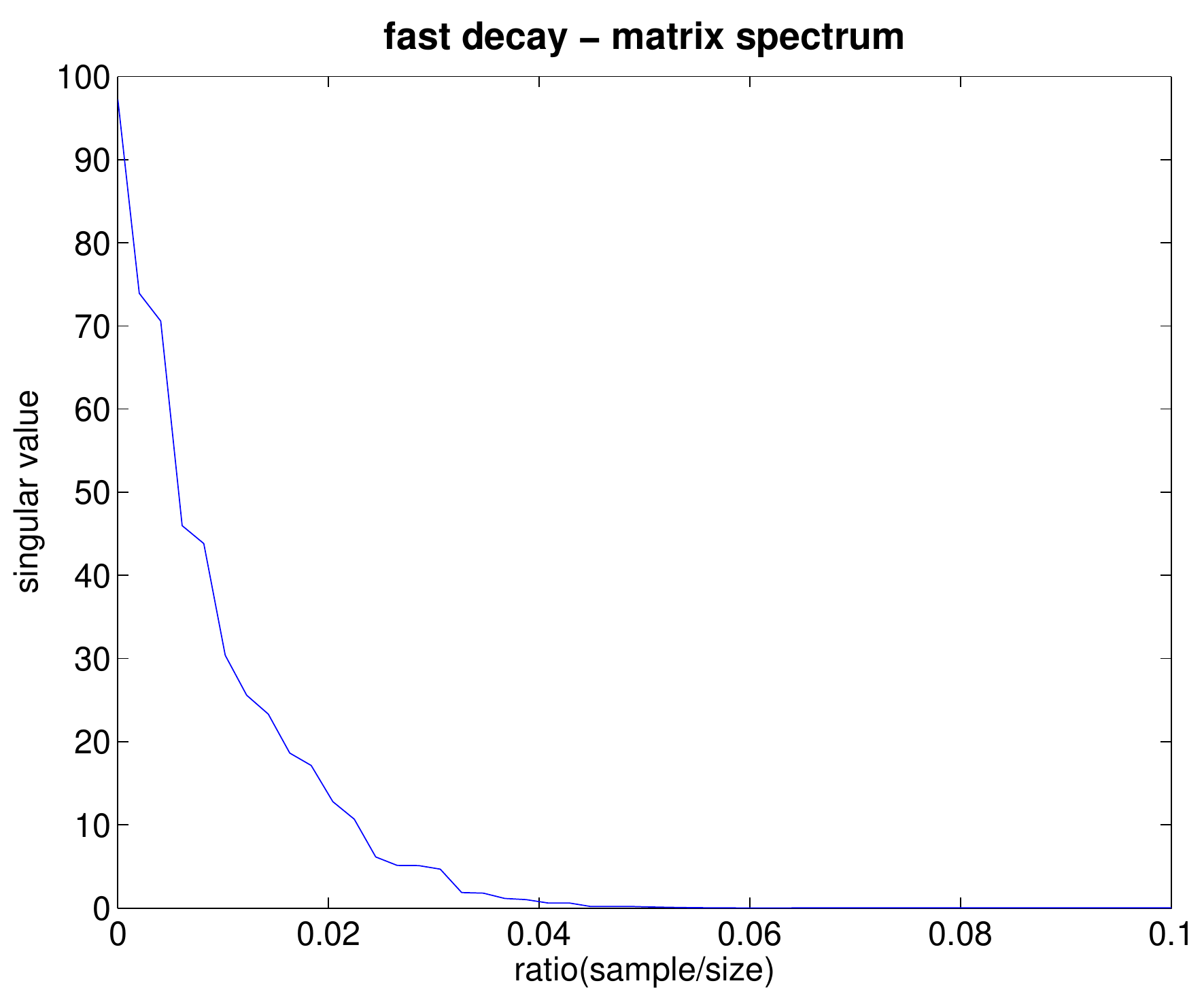}\tabularnewline
\end{tabular}
    \end{center}
    \vspace{-20pt}
    \caption{Nystr\"{o}m approximation errors for random matrices. The X-axis is the sampling ratio given as sample size divided by the matrix size.
    The Y-axis is the approximation error given in $L_2$ norm. The tested algorithms are Random, LinearTimeSVD, Algorithm 1 and SVD. }
    \label{fig:general_approx_errors}
\end{figure}

\subsection{ Non-Singularity of Sample Matrix }

We empirically examine the relationship between the Nystr\"{o}m
approximation error and the non-singularity of the sub-sample
matrix. The approximation error is measured in $L_2$-norm and the
non-singularity of $A_M\in\mathbb{R}^{s\times s}$ is measured by the
magnitude of $\sigma_s\left(A_M\right)$. We employ testing matrices
similar to those in section \ref{sec:general_mat_exp}. These feature
a non-random spectrum and random singular subspaces. The sample was
chosen to be 5\% of the data of the matrix. In this test, we compare
between the random sample selection algorithm and Algorithm 1.
Each algorithm ran 100 times on each matrix. The
results of each run were recorded. Figure \ref{fig:non_sing_of_a_m}
features a log-log scale plot of the approximation error as a
function of $\sigma_s\left(A_M\right)$. The performance of the
different algorithm versions is compared. We arrive at similar
conclusions to those in section \ref{sec:general_mat_exp}. Our
algorithms do no better than random sampling when the spectrum decay
is slow, but consistently outperforms the random selection in the
presence of fast spectrum decay. Figure \ref{fig:non_sing_of_a_m}
also shows a strong negative correlation between the variables in
all the examined matrices. Hence, a large $\sigma_s\left(A_M\right)$
implies a small approximation error. The linear shape of the graphs,
drawn in a log-log scale, suggests that this relationship is
exponential. The results hint at a possible extension of the
Nystr\"{o}m procedure to a Monte-Carlo method: Algorithm 1 can be
run many times. In the end, we choose the sample for which
$\sigma_s\left(A_M\right)$ is maximal.

\begin{figure}[h]
    \begin{center}
\begin{tabular}{cc}
\includegraphics[width=0.47\columnwidth]{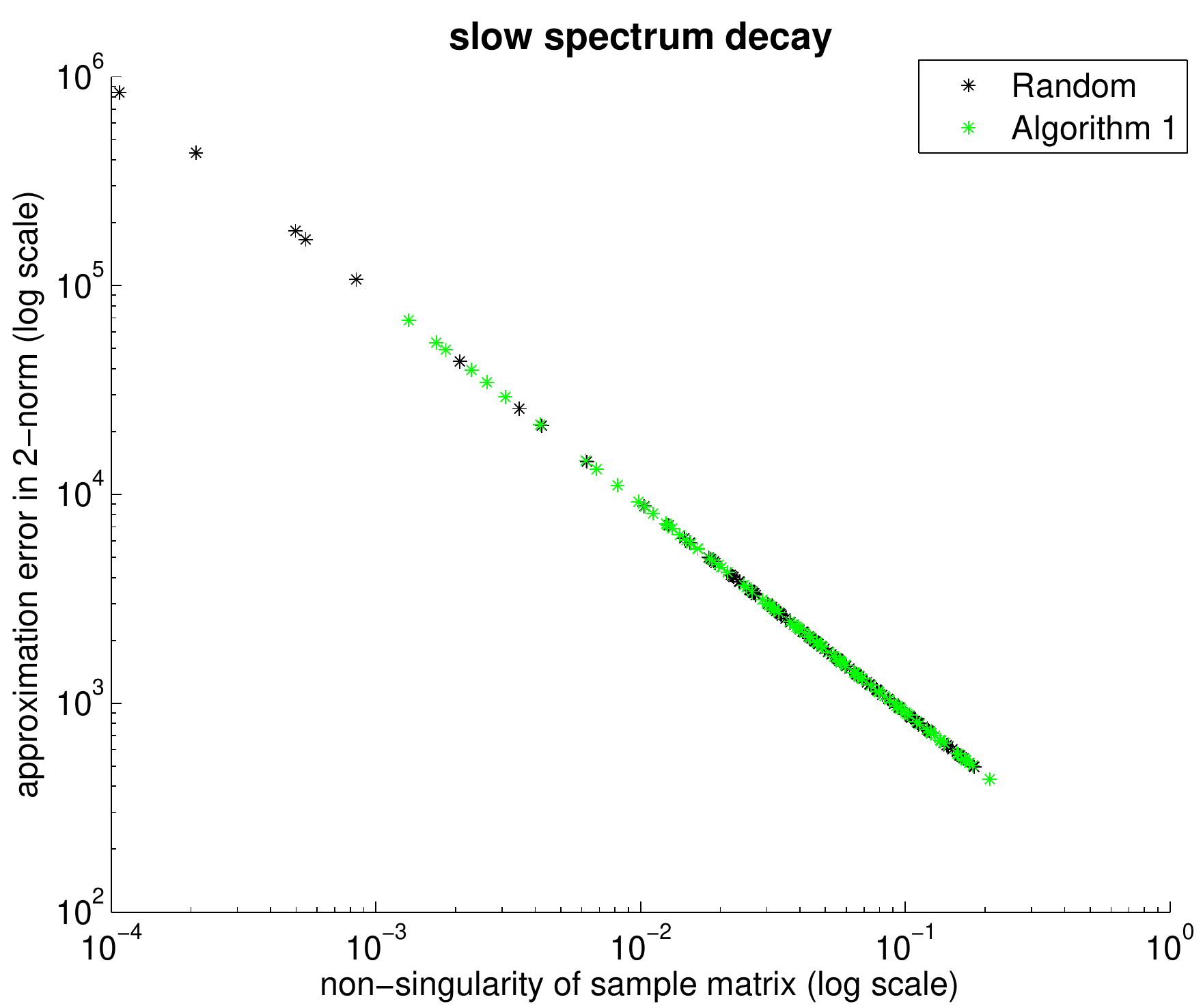} & \includegraphics[width=0.475\columnwidth]{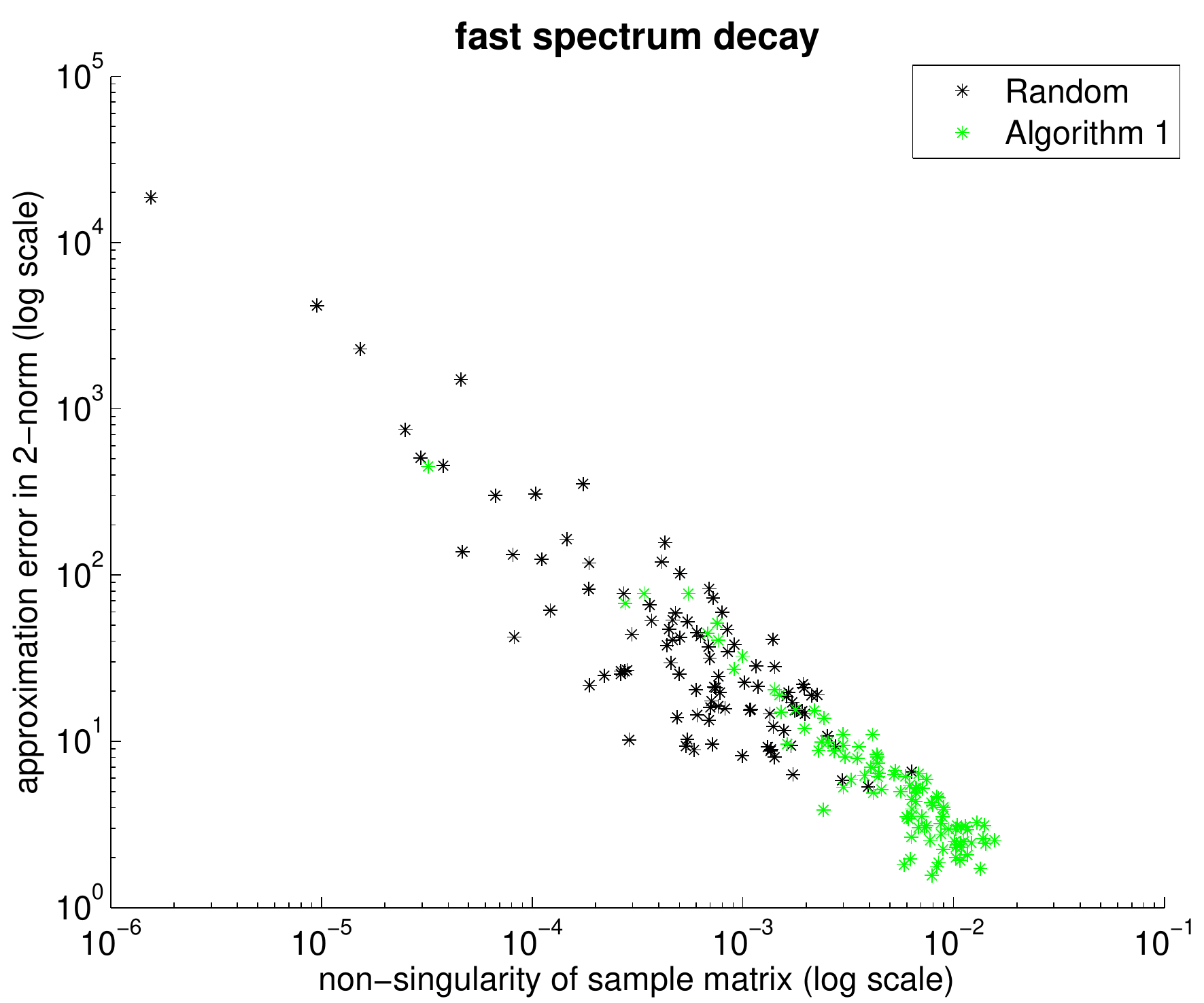}\tabularnewline
\includegraphics[width=0.47\columnwidth]{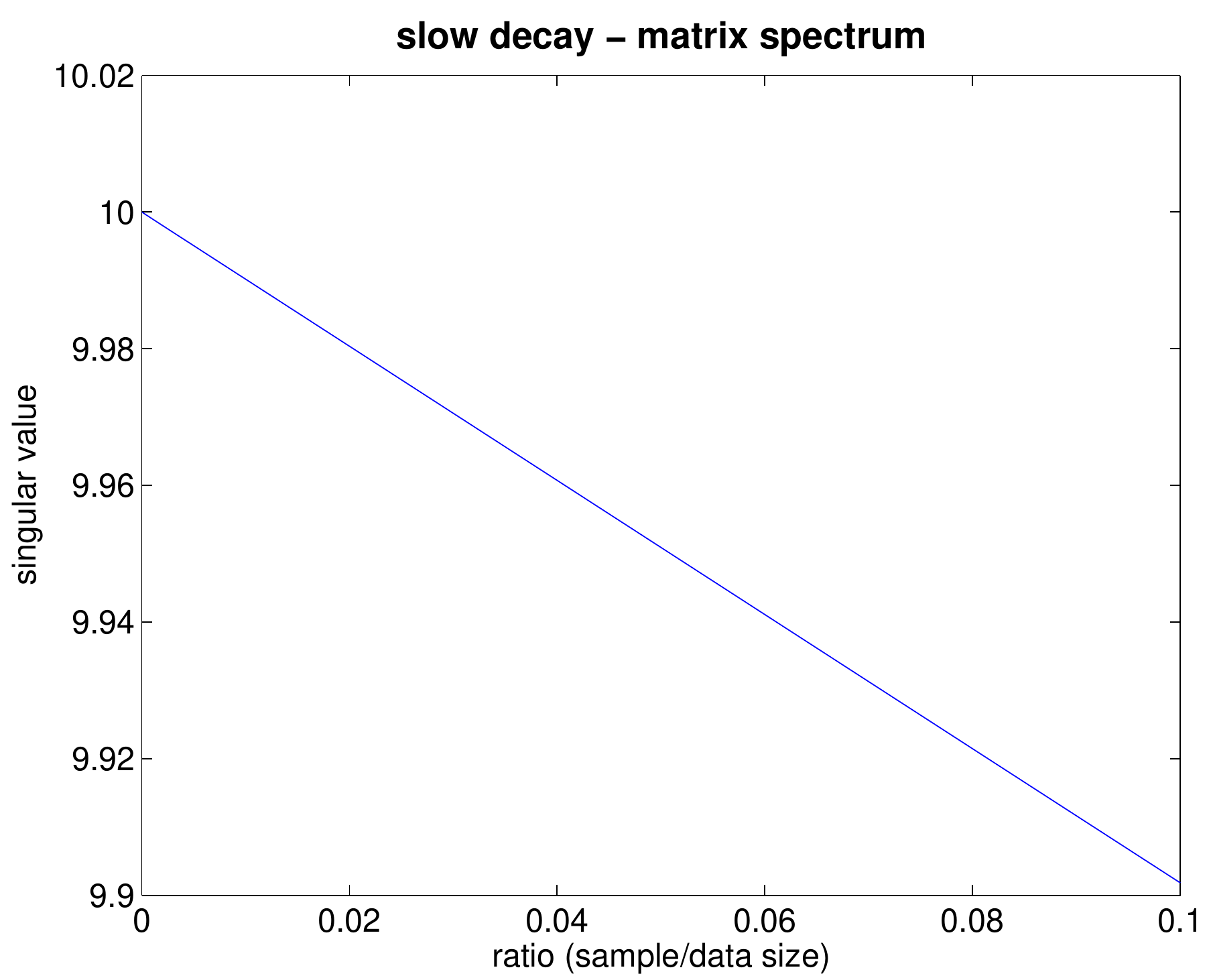} & \includegraphics[width=0.455\columnwidth]{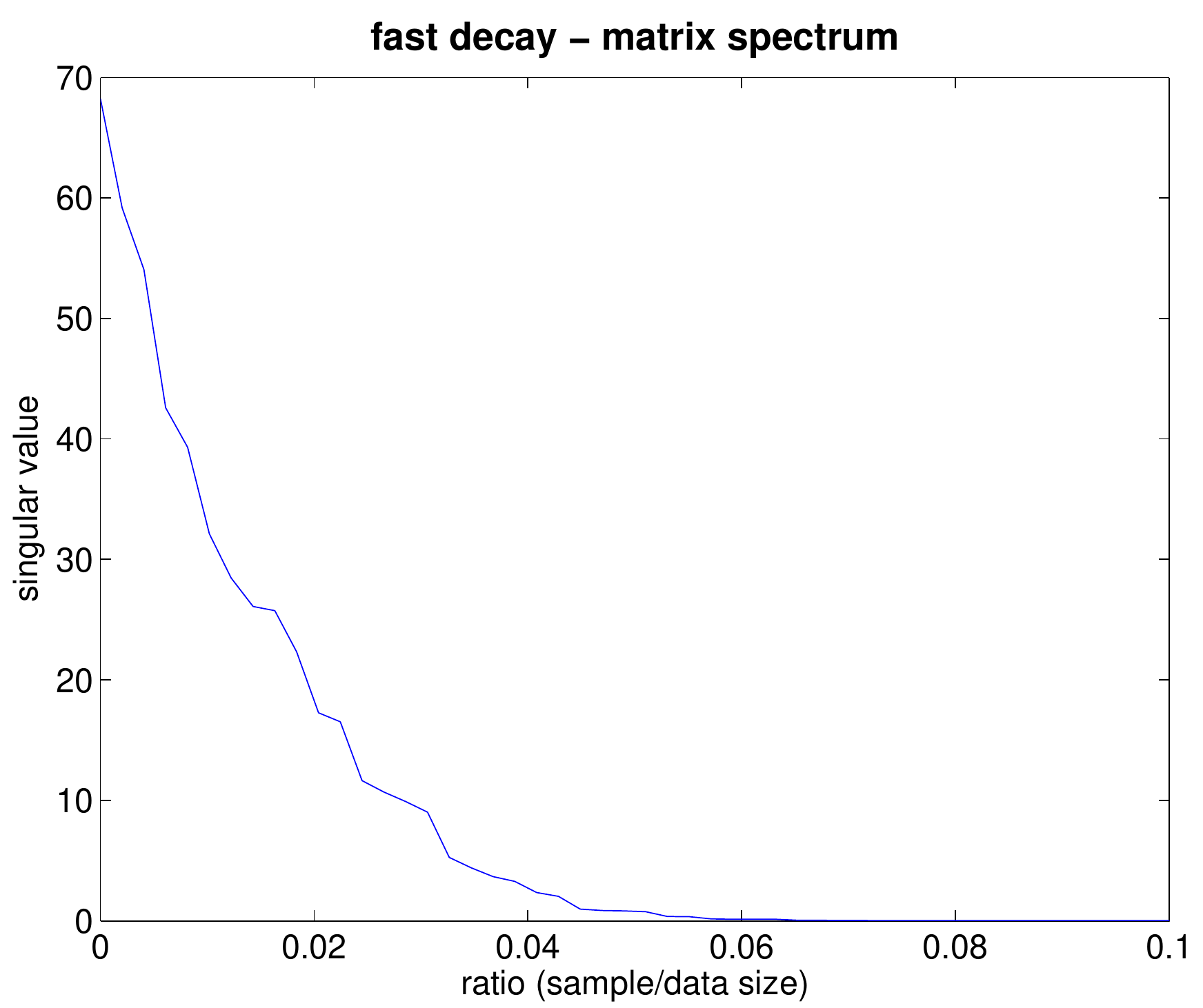}\tabularnewline
\end{tabular}
    \end{center}
    \vspace{-20pt}
    \caption{ Errors in Nystr\"{o}m approximation as a function of $\sigma_s\left(A_M\right)$ }
    \label{fig:non_sing_of_a_m}
\end{figure}

\section{ Conclusion and Future Research }
\label{sec:conclusions}

In this paper, we showed how the Nystr\"{o}m approximation method
can be used to find the canonical SVD and EVD of a general matrix.
In addition, we developed a sample selection algorithm that operates
on general matrices. Experiments have been performed on real-world
kernels and random general matrices. These show that the algorithm
performs well when the spectrum of the matrix decays quickly and the
sample is sufficiently large to capture most of the energy of the
matrix (the number of non-zero singular values). Another experiment
showed that the non-singularity of the sample matrix (as measured by
the magnitude of the smallest singular value) is exponentially
inversely related to the approximation error. This shows that our
theoretical reasoning in Lemma \ref{lem:span_approx_error} is
qualitatively on par with empirical evidence.

Future research should focus on additional formalization of the
relationship between the smallest singular value of the sample
matrix and the Nystr\"{o}m approximation error. Another interesting
possibility is to find a constrained class of matrices and develop a
sample selection algorithm to take advantage of the constraint. Some
classes of matrices may be easier to sub-sample with respect to the
Nystr\"{o}m method.

%\par\noindent
%\section*{Acknowledgment} 
%The second author in this research was supported by
%the Israel Science Foundation (Grant No.  1041/10).
%
%\section*{Bibliography}

\end{document}